\documentclass[11pt,a4paper,final]{article}
\usepackage{amsmath}
\usepackage{amssymb}
\usepackage{amsxtra}
\usepackage{amsthm}
\usepackage{bbm}			
\usepackage{cite}

\usepackage[left=3.1cm,right=3.1cm,top=3.3cm,bottom=3.3cm]{geometry}

\usepackage[usenames,dvipsnames]{xcolor}
\usepackage[plainpages=false,pdfpagelabels,colorlinks=true,linkcolor=blue,citecolor=Magenta]{hyperref} 

\usepackage{showkeys}
\usepackage[latin1]{inputenc}       
\usepackage[T1]{fontenc}        
\usepackage{lmodern}            
\usepackage{graphicx}
\usepackage{subfig}		
\usepackage{booktabs}		
\usepackage{multirow}
\usepackage{enumitem}		
\usepackage{listings} 	
\usepackage{mathtools}  
\usepackage{floatrow}		
\usepackage{caption}

\usepackage{colortbl}
\definecolor{dunkelgrau}{rgb}{0.8,0.8,0.8}
\definecolor{hellgrau}{rgb}{0.9,0.9,0.9}

\newcommand{\R}{\ensuremath{\mathbb{R}}}
\newcommand{\N}{\ensuremath{\mathbb{N}}}

\newcommand{\oR}{\ensuremath{\overline{\mathbb{R}}}}

\newcommand{\D}{\ensuremath{\mathcal{D}}}

\renewcommand{\>}{\right\rangle}
\newcommand{\<}{\left\langle}
\newcommand{\id}{\ensuremath{\text{Id}}}

\newcommand{\bx}{\ensuremath{\overline{x}}}
\newcommand{\by}{\ensuremath{\overline{y}}}
\newcommand{\bz}{\ensuremath{\overline{z}}}
\newcommand{\bv}{\ensuremath{\overline{v}}}

\newcommand{\bq}{\ensuremath{\overline{q}}}
\newcommand{\bp}{\ensuremath{\overline{p}}}

\renewcommand{\bv}{\ensuremath{\overline{v}}}

\newcommand{\f}{\ensuremath{\boldsymbol}}
\newcommand{\wt}{\ensuremath{\widetilde}}
\newcommand{\fx}{\ensuremath{\boldsymbol{x}}}
\newcommand{\fv}{\ensuremath{\boldsymbol{v}}}

\newcommand{\fq}{\ensuremath{\boldsymbol{q}}}

\newcommand{\fy}{\ensuremath{\boldsymbol{y}}}

\newcommand{\fK}{\ensuremath{\boldsymbol{\mathcal{K}}}}

\newcommand{\I}{\ensuremath{\mathcal{I}}}

\newcommand{\g}{\ensuremath{\mathcal{G}}}
\newcommand{\h}{\ensuremath{\mathcal{H}}}
\newcommand{\X}{\ensuremath{\mathcal{X}}}
\newcommand{\Y}{\ensuremath{\mathcal{Y}}}
\newcommand{\proj}{\ensuremath{\mathcal{P}}}

\renewcommand{\Box}{\ensuremath{\mbox{\small$\,\square\,$}}}

\makeatletter
\renewcommand{\@fnsymbol}[1]{\@arabic{#1}}  
\makeatother

\theoremstyle{plain}
\newtheorem{theorem}{Theorem}[section]

\newtheorem{proposition}[theorem]{Proposition}

\theoremstyle{definition}
\newtheorem{remark}{Remark}[section]

\newtheorem{problem}{Problem}[section]
\newtheorem{algorithm}{Algorithm}[section]

\DeclareMathOperator*\dom{dom}%
\DeclareMathOperator*\ri{ri}%
\DeclareMathOperator*\sqri{sqri}%
\DeclareMathOperator*\gra{gra}%
\DeclareMathOperator*\argmin{arg\,min}%
\DeclareMathOperator*\cl{cl}%
\DeclareMathOperator*\ran{ran}%
\DeclareMathOperator*\zer{zer}%
\DeclareMathOperator*\Prox{Prox}%

\numberwithin{equation}{section}  

\title{Solving monotone inclusions involving parallel sums of linearly composed maximally monotone operators}

\author{Radu Ioan Bo\c t
\thanks {University of Vienna, Faculty of Mathematics, Nordbergstra\ss e 15, A-1090 Vienna, Austria, e-mail: radu.bot@univie.ac.at. Research partially supported by DFG (German Research Foundation), project BO 2516/4-1.}
\and Christopher Hendrich
\thanks{Chemnitz University of Technology, Department of Mathematics, D-09107 Chemnitz, Germany, e-mail: christopher.hendrich@mathematik.tu-chemnitz.de. Research supported by a Graduate Fellowship of the Free State Saxony, Germany.}
}
\date{\today}

\begin{document}
\maketitle

{\bf Abstract.} The aim of this article is to present two different primal-dual methods for solving structured monotone inclusions involving parallel sums of compositions of maximally monotone operators with linear bounded operators. By employing some elaborated splitting techniques, all of the operators occurring in the problem formulation are processed individually via forward or backward steps. The treatment of parallel sums of linearly composed maximally monotone operators is motivated by applications in imaging which involve first- and 
second-order total variation functionals, to which a special attention is given.



\section{Introduction}\label{ic_sectionIntro}

In applied mathematics, a wide variety of convex optimization problems such as single- or multifacility location problems, support vector machine problems for classification and regression, problems in clustering and portfolio optimization as well as signal and image processing problems, all of them potentially possessing nonsmooth terms in their objectives, can be reduced to the solving of inclusion problems involving mixtures of monotone set-valued operators.

Therefore, the solving of monotone inclusion problems involving maximally monotone operators (see \cite{AttBriCom10,  BecCom13, BotCseHei12, BotCseHei12b, BotHend12d, BotHend12e, BriCom11, ChaPoc11, Com01, Com04, Com09,Com12,ComPes12, Con12,DouRac56, EckBer92,Roc76,SetSteTeu11,StaPesSte13,Tse00,Vu13}) continues to be one of the most attractive branches of research. To the most popular methods for solving monotone inclusion problems belong the proximal point algorithm (see \cite{Roc76}) and the Douglas-Rachford splitting algorithm (see \cite{DouRac56}).

In the last years, motivated by different applications, the complexity of the monotone inclusion problems increased by allowing in their formulation maximally monotone operators composed with linear bounded operators (see \cite{BriCom11,ChaPoc11}), (single-valued) Lipschitzian or cocoercive monotone operators and parallel sums of maximally monotone operators (see \cite{BecCom13, BotCseHei12, BotHend12e,Com12,ComPes12,Con12,Vu13}). Also, under strong monotonicity assumptions, for some of these iterative schemes accelerated versions have been provided (see \cite{ChaPoc11,BotCseHei12b,BotHend12d}).

Our problem formulation is inspired by a real-world application in imaging (cf. \cite{ChaLio97, SetSteTeu11}), where first- and second-order total variation functionals are linked via infimal convolutions in order to reduce staircasing effects in the reconstructed images. The problem under investigation follows.

\begin{problem}\label{ic_problem}
Let $\h$ be a real Hilbert space, $z \in \h$, let $A:\h \rightarrow 2^{\h}$ be a maximally monotone operator, and $C:\h \rightarrow \h$ be a monotone $\mu^{-1}$-cocoercive operator for $\mu\in\R_{++}$.  Furthermore, for every $i=1,\ldots,m$, let $\g_i$, $\X_i$, $\Y_i$ be real Hilbert spaces, $r_i\in \g_i$, $B_i : \X_i \rightarrow 2^{\X_i}$ and $D_i:\Y_i \rightarrow 2^{\Y_i}$ be maximally monotone operators and consider the nonzero linear bounded operators $L_i:\h \rightarrow \g_i$, $K_i:\g_i \rightarrow \X_i$ and $M_i:\g_i \rightarrow \Y_i$. The problem is to solve the primal inclusion
\begin{align}
	\label{ic_p_primal}
	\text{find }\bx \in \h \text{ such that } \!z \in A\bx + \!\sum_{i=1}^mL_i^* \Big(\big(K_i^*\circ B_i\circ K_i \big)\Box \big(M_i^*\circ D_i\circ M_i\big)\Big)(L_i \bx-r_i) + C\bx
\end{align}
together with its dual inclusion
\begin{align}
	\label{ic_p_dual}
\text{find } \!\!\left\{\!\!\!\begin{array}{l}
\bp_i \in \X_i, i=1,...,m,\\
\bq_i \in \Y_i, i=1,...,m,\\
\by_i \in \g_i, i=1,...,m,
\end{array}\right. \!\!\!\text{such that } \!\exists x\in\h: \!\left\{
	\begin{array}{l}
		z - \sum_{i=1}^mL_i^*K_i^*\bp_i \in Ax +Cx, \\
        K_i(L_ix -\by_i-r_i) \in B_i^{-1}\bp_i, i=1,...,m,\\
		M_i\by_i \in D_i^{-1}\bq_i, i=1,...,m,\\
        K_i^*\bp_i = M_i^*\bq_i, i=1,...,m.
	\end{array}
\right.
\end{align}
\end{problem}

We provide in this paper two iterative methods of forward-backward and forward-backward-forward type, respectively, for solving this primal-dual pair of monotone inclusion problems and investigate their asymptotic behavior. 
A very similar problem formulation was recently investigated in \cite{BecCom13}, however, the proposed iterative scheme there relies on the forward-backward-forward method and is 
different from the corresponding one which we propose here. However, since it makes a forward step less, the forward-backward method is more attractive from the perspective of its numerical implementation. This phenomenon is 
supported by our experimental results reported in Section \ref{ic_secExperiment}.

The article is organized as follows. In Section \ref{ic_secNotation} we introduce notations and preliminary results in convex analysis and monotone operator theory. 
In Section \ref{ic_secMain} we formulate the two algorithms and study their convergence behavior. In Section \ref{ic_secConvexMinimization} we employ the outcomes of the previous one to the simultaneously solving of 
convex minimization problems and their conjugate dual problems. Numerical experiments in the context of image denoising problems with first- and second-order total variation functionals are made in Section \ref{ic_secExperiment}.

\section{Notation and preliminaries}\label{ic_secNotation}
We are considering the real Hilbert space $\h$ endowed with an \textit{inner product} $\left\langle \cdot ,\cdot \right\rangle$ and associated \textit{norm} $\left\| \cdot \right\| = \sqrt{\left\langle \cdot, \cdot \right\rangle}$. The symbols $\rightharpoonup$ and $\rightarrow$ denote weak and strong convergence, respectively. Having the sequences $(x_n)_{n\geq 0}$ and $(y_n)_{n\geq 0}$ in $\h$, we mind errors in the implementation of the algorithm by using the following notation taken from \cite{BecCom13}
\begin{align} \label{ic_errors}
	\left( x_n \approx y_n \ \forall n \geq 0 \right) \Leftrightarrow \sum_{n \geq 0} \|x_n-y_n\| < +\infty.
\end{align}
By $\R_{++}$ we denote the set of strictly positive real numbers and by $\R_+ := \R_{++} \cup \{0\}$. For a function $f: \h \rightarrow \oR := \R \cup \{\pm \infty\}$ we denote by $\dom f := \left\{ x \in \h : f(x) < +\infty \right\}$ its \textit{effective domain} and call $f$ \textit{proper}, if $\dom f \neq \varnothing$ and $f(x)>-\infty$ for all $x \in \h$. Let be
$$\Gamma(\h) := \{f: \h \rightarrow \overline \R: f \ \mbox{is proper, convex and lower semicontinuous}\}.$$
The \textit{conjugate function} of $f$ is $f^*:\h \rightarrow \oR$, $f^*(p)=\sup{\left\{ \left\langle p,x \right\rangle -f(x) : x\in\h \right\}}$ for all $p \in \h$ and, if $f \in \Gamma(\h)$, then $f^* \in \Gamma(\h)$, as well. The \textit{(convex) subdifferential} of $f: \h \rightarrow \oR$ at $x \in \h$ is the set $\partial f(x) = \{p \in \h : f(y) - f(x) \geq \left\langle p,y-x \right\rangle \ \forall y \in \h\}$, if $f(x) \in \R$, and is taken to be the empty set, otherwise. For a linear bounded operator $L: \h \rightarrow \g$, where $\g$ is another real Hilbert space, the operator $L^*: \g \rightarrow \h$, defined via $\< Lx,y  \> = \< x,L^*y  \>$ for all $x \in \h$ and all $y \in \g$, denotes its \textit{adjoint}.

Having two proper functions $f,\,g : \h \rightarrow \oR$, their \textit{infimal convolution} is defined by $f \Box g : \h \rightarrow \oR$, $(f \Box g) (x) = \inf_{y \in \h}\left\{ f(y) + g(x-y) \right\}$ for all $x \in \h$, being a convex function when $f$ and $g$ are convex.

Let $M:\h \rightarrow 2^{\h}$ be a set-valued operator. We denote by $\zer M = \{ x \in \h : 0 \in M x \}$ its set of \textit{zeros}, by $\gra M = \{ (x,u) \in \h \times \h : u \in Mx\}$ its \textit{graph} and by $\ran M =\{u \in \h : \exists x \in \h,\ u\in Mx\}$ its \textit{range}. The \textit{inverse} of $M$ is $M^{-1}:\h \rightarrow 2^{\h}$, $u \mapsto \{ x\in\h : u \in Mx \}$. We say that the operator $M$ is \textit{monotone} if $\< x-y,u-v \> \geq 0$ for all $(x,u),\,(y,v) \in \gra M$ and it is said to be \textit{maximally monotone} if there exists no monotone operator $M':\h \rightarrow 2^{\h}$ such that $\gra M'$ properly contains $\gra M$. The operator $M$ is said to be \textit{uniformly monotone} with modulus $\phi_M : \R_{+} \rightarrow [0,+\infty]$ if $\phi_M$ is increasing, vanishes only at $0$, and $\< x-y,u-v \> \geq \phi_M \left( \| x-y \|\right)$ for all $(x,u),\,(y,v) \in \gra M$.

Let $\mu>0$ be arbitrary. A single-valued operator $M: \h \rightarrow \h$ is said to be \textit{$\mu$-cocoercive} if $\langle x-y,Mx-My\rangle\geq \mu\|Mx-My\|^2$ for all $(x,y)\in \h \times \h$. Moreover, $M$ is \textit{$\mu$-Lipschitzian} if $\|Mx-My\|\leq \mu\|x-y\|$ for all $(x,y)\in \h \times \h$. A linear bounded operator $M:\h \rightarrow \h$ is said to be \textit{self-adjoint}, if $M=M^*$ and \textit{skew}, if $M^* = -M$. 

The \textit{sum} and the \textit{parallel sum} of two set-valued operators $M_1,\,M_2:\h \rightarrow 2^{\h}$ are defined as
$M_1 + M_2:\h \rightarrow 2^{\h}, (M_1 + M_2)(x) = M_1(x) + M_2(x) \ \forall x \in \h$
and
$$M_1 \Box M_2:\h \rightarrow 2^{\h}, M_1 \Box M_2  = \left(M_1^{-1} + M_2^{-1}\right)^{-1},$$
respectively. If $M_1$ and $M_2$ are monotone, than $M_1 + M_2$ and $M_1 \Box M_2$ are monotone, too. However, if $M_1$ and $M_2$ are maximally monotone, this property is in general not true neither for $M_1 + M_2$ nor for $M_1 \Box M_2$, unless some qualification conditions are fulfilled (see \cite{BauCom11, Bot10, Zal02}).

The \textit{resolvent} of an operator  $M:\h \rightarrow 2^{\h}$ is
$$ J_M = \left( \id + M \right)^{-1},$$
the operator $\id$ denoting the \textit{identity} on the underlying Hilbert space. When $M$ is maximally monotone, its resolvent is a single-valued firmly nonexpansive operator and, by \cite[Proposition 23.18]{BauCom11}, we have for $\gamma \in \R_{++}$
\begin{align}
	\label{ic_res-identity}
	\id = J_{\gamma M} + \gamma J_{\gamma^{-1}M^{-1}}\circ \gamma^{-1} \id.
\end{align}
Moreover, for $f \in \Gamma(\h)$ and $\gamma \in \R_{++}$ the subdifferential $\partial (\gamma f)$ is maximally monotone (cf. \cite{Roc70}) and it holds $J_{\gamma \partial f} = \left(\id + \gamma \partial f \right)^{-1} = \Prox_{\gamma f}$. Here, $\Prox_{\gamma f}(x)$ denotes the \textit{proximal point} of $\gamma f$ at $x\in\h$, representing the unique optimal solution of the optimization problem
\begin{align}\label{ic_prox-def}
\inf_{y\in \h}\left \{\gamma f(y)+\frac{1}{2}\|y-x\|^2\right\}.
\end{align}
In this particular situation, relation \eqref{ic_res-identity} becomes \textit{Moreau's decomposition formula}
\begin{align}
	\label{ic_res-identity_functions}
	\id = \Prox\nolimits_{\gamma f} + \gamma \Prox\nolimits_{\gamma^{-1}f^*} \circ \gamma^{-1}\id.
\end{align}
When $\Omega \subseteq \h$ is a nonempty, convex and closed set, the function $\delta_\Omega : \h \rightarrow \overline \R$, defined by $\delta_\Omega(x) = 0$ for $x \in \Omega$ and $\delta_\Omega(x) = +\infty$, otherwise, denotes the \textit{indicator function} of the set $\Omega$. For each $\gamma > 0$ the proximal point of $\gamma \delta_\Omega$ at $x \in \h$ is nothing else than
$$\Prox\nolimits_{\gamma \delta_\Omega}(x) = \Prox\nolimits_{\delta_\Omega}(x) = \proj_\Omega(x) = \argmin_{y \in \Omega} \frac{1}{2}\|y-x\|^2,$$
which is the \textit{projection} of $x$ on $\Omega$.

Finally, when for $i=1,\ldots,m$ the real Hilbert spaces $\h_i$ are endowed with inner product $\left\langle \cdot ,\cdot \right\rangle_{\h_i}$ and associated norm $\left\| \cdot \right\|_{\h_i} = \sqrt{\left\langle \cdot, \cdot \right\rangle_{\h_i}}$, we denote by
$$\f \h=\h_1\oplus\ldots\oplus\h_m$$
their direct sum. For $\fv=(v_1,\ldots,v_m)$, $\fq=(q_1,\ldots,q_m) \in \f \h$, this real Hilbert space is endowed with inner product and associated norm defined via
\begin{align*}
	\< \fv,\fq \>_{\f \h} = \sum_{i=1}^m \< v_i,q_i \>_{\h_i}  \text{ and, respectively, } \|\fv\|_{\f \h} = \sqrt{\sum_{i=1}^m \| v_i \|_{\h_i}^2}.
\end{align*}

\section{The primal-dual iterative schemes}\label{ic_secMain}
Within this section we provide two different algorithms for solving the primal-dual inclusion introduced in Problem \ref{ic_problem} and discuss their asymptotic convergence. 
In Subsection \ref{ic_subsecFBF}, however, the assumptions imposed on the monotone operator $C:\h \rightarrow \h$ are weakened by assuming that $C$ is only $\mu$-Lipschitz continuous for some $\mu \in \R_{++}$.

In the following let be
\begin{align*}
\f \X = \X_1 \oplus \ldots \oplus \X_m, \ \f \Y=\Y_1\oplus \ldots \oplus \Y_m,\ \f \g=\g_1\oplus\ldots\oplus\g_m
\end{align*}
and
\begin{align*}
\f p=(p_1,\ldots,p_m) \in \f \X,\ \f q=(q_1,\ldots,q_m) \in \f \Y,\ \f y=(y_1,\ldots,y_m) \in \f \g.
\end{align*}
We say that $(\bx, \f \bp, \f \bq, \f \by) \in \h \oplus \f \X \oplus \f \Y \oplus \f \g$ is a primal-dual solution to Problem \ref{ic_problem}, if
\begin{align}
\label{ic_condition}
	\begin{aligned}
& z - \sum_{i=1}^mL_i^*K_i^*\bp_i \in A\bx +C\bx \ \mbox{and} \\
& K_i(L_i\bx -\by_i-r_i) \in B_i^{-1}\bp_i,\ M_i\by_i \in D_i^{-1}\bq_i,\  K_i^*\bp_i = M_i^*\bq_i,\ i=1,\ldots,m.
	\end{aligned}
\end{align}
If  $(\bx, \f \bp, \f \bq, \f \by) \in \h \oplus \f \X \oplus \f \Y \oplus \f \g$ is a primal-dual solution to Problem \ref{ic_problem}, then $\bx$ is a solution to \eqref{ic_p_primal} and $(\f \bp, \f \bq, \f \by)$ is a solution to \eqref{ic_p_dual}. Notice also that
{\allowdisplaybreaks
\begin{align}\label{ic_equivsol}
&\bx \text{ solves }\eqref{ic_p_primal}  \Leftrightarrow z   \in A\bx+ \sum_{i=1}^mL_i^* \Big(\big(K_i^*\circ B_i\circ K_i \big)\Box \big(M_i^*\circ D_i\circ M_i\big)\Big)(L_i \bx-r_i)+ C\bx \nonumber \\
\Leftrightarrow \ &\exists\, \f \bv \in \f \g \text{ such that} \left\{
		\begin{array}{l} z - \sum_{i=1}^mL_i^*\bv_i  \in A\bx + C\bx, \\
L_i \bx-r_i \in  \big(K_i^*\circ B_i\circ K_i \big)^{-1}(\bv_i) + \big(M_i^*\circ D_i\circ M_i\big)^{-1}(\bv_i),\\
 i=1,\ldots,m.  \end{array}\right. \nonumber \\
\Leftrightarrow \ &\exists\, (\f \bv, \f \by) \in \f \g \oplus \f \g \text{ such that} \left\{
		\begin{array}{l} z - \sum_{i=1}^mL_i^*\bv_i  \in A\bx + C\bx, \\
\bv_i \in \big(K_i^*\circ B_i\circ K_i \big)(L_i \bx - \by_i - r_i),\ i=1,\ldots,m,\\
\bv_i \in  \big(M_i^*\circ D_i\circ M_i\big)(\by_i),\ i=1,\ldots,m
\end{array}\right. \nonumber \\
\Leftrightarrow \ &\exists\, (\f \bp, \f \bq, \f \by) \in \f \X \oplus \f \Y \oplus \f \g \text{ such that} \left\{
		\begin{array}{l} z - \sum_{i=1}^mL_i^*K_i^*\bp_i  \in A\bx + C\bx, \\
\bp_i \in \big(B_i\circ K_i \big)(L_i \bx - \by_i - r_i),\ i=1,\ldots,m,\\
\bq_i \in  \big(D_i\circ M_i\big)(\by_i),\ i=1,\ldots,m,\\
K_i^*\bp_i = M_i^*\bq_i,\ i=1,\ldots,m.
\end{array}\right. \nonumber \\
\Leftrightarrow \ &\exists\, (\f \bp, \f \bq, \f \by) \in \f \X \oplus \f \Y \oplus \f \g \text{ such that} \left\{
		\begin{array}{l}
		z - \sum_{i=1}^mL_i^*K_i^*\bp_i \in A\bx +C\bx, \\
        K_i(L_i\bx -\by_i-r_i) \in B_i^{-1}\bp_i,\, i=1,\ldots,m,\\
		M_i\by_i \in D_i^{-1}\bq_i,\ i=1,\ldots,m,\\
        K_i^*\bp_i = M_i^*\bq_i,\ i=1,\ldots,m.
	\end{array}\right.
\end{align}}
Thus, if $\bx$ is a solution to \eqref{ic_p_primal}, then there exists $(\f \bp, \f \bq, \f \by) \in \f \X \oplus \f \Y \oplus \f \g$ such that $(\bx, \f \bp, \f \bq, \f \by)$ is a primal-dual solution to Problem \ref{ic_problem} and if $(\f \bp, \f \bq, \f \by)$ is a solution to \eqref{ic_p_dual}, then there exists $\bx \in \h$ such that $(\bx, \f \bp, \f \bq, \f \by)$ is a primal-dual solution to Problem \ref{ic_problem}.

\begin{remark}\label{ic_remark_errors}
The notations  \eqref{ic_errors} have been introduced in order to allow errors in the implementation of the algorithm, without affecting the readability of the paper in the sequel. This is reasonable since errors preserve their summability under addition, scalar multiplication and linear bounded mappings.
\end{remark}

\subsection{An algorithm of forward-backward type}\label{ic_subsecFB}
In this subsection we propose a forward-backward type algorithm for solving Problem \ref{ic_problem} and prove its convergence by showing that it can be reduced to an error-tolerant forward-backward iterative scheme.

\begin{algorithm}\label{ic_alg1} \text{ }\newline
Let $x_0 \in \h$, and for any $i=1,\ldots,m$, let $p_{i,0}\in\X_i$, $q_{i,0} \in \Y_i$ and $z_{i,0},\ y_{i,0},\ v_{i,0} \in \g_i$. For any $i=1,\ldots,m$, let $\tau,\,\theta_{1,i},\, \theta_{2,i},\, \gamma_{1,i},\, \gamma_{2,i}$ and $\sigma_i$ be strictly positive real numbers such that
\begin{align}
	\label{ic_condition_step_length}
	2\mu^{-1}\left(1-\overline{\alpha} \right) \min_{i=1,\ldots,m}\bigg\{\frac{1}{\tau},\frac{1}{\theta_{1,i}},\frac{1}{\theta_{2,i}}, \frac{1}{\gamma_{1,i}},\frac{1}{\gamma_{2,i}},\frac{1}{\sigma_i} \bigg\}>1,
\end{align}
for
\begin{align*}
\overline{\alpha} = \max\left\{\sqrt{\tau \sum_{i=1}^m \sigma_i \|L_i\|^2}, \max_{j=1,\ldots,m}\left\{\sqrt{\theta_{1,j}\gamma_{1,j}\|K_j\|^2}, \sqrt{\theta_{2,j}\gamma_{2,j}\|M_j\|^2}\right\} \right\}.
\end{align*}
Furthermore, let $\varepsilon \in (0,1)$, $(\lambda_n)_{n\geq 0}$ a sequence in $\left[\varepsilon,1\right]$ and set
{\allowdisplaybreaks	
\begin{align}\label{ic_A1}
	  \left(\forall n\geq 0\right) \begin{array}{l} \left\lfloor \begin{array}{l}
		\wt x_{n} \approx J_{\tau A}\left( x_n - \tau\left(Cx_n + \sum_{i=1}^m L_i^* v_{i,n} - z\right) \right) \\
		\text{For }i=1,\ldots,m  \\
				\left\lfloor \begin{array}{l}
					\wt p_{i,n} \approx J_{\theta_{1,i} B_i^{-1}}\left(p_{i,n} +\theta_{1,i} K_i z_{i,n}  \right) \\
					\wt q_{i,n} \approx J_{\theta_{2,i} D_i^{-1}}\left(q_{i,n} +\theta_{2,i} M_i y_{i,n}  \right) \\
					u_{1,i,n} \approx z_{i,n} +\gamma_{1,i}\left(K_i^*\left(p_{i,n}-2\wt p_{i,n}\right)+v_{i,n}+\sigma_i\left(L_i(2\wt x_{n}-x_n)-r_i\right)\right) \\
					u_{2,i,n} \approx y_{i,n} +\gamma_{2,i}\left(M_i^*\left(q_{i,n}-2\wt q_{i,n}\right)+v_{i,n}+\sigma_i\left(L_i(2\wt x_{n}-x_n)-r_i\right)\right) \\
					\wt z_{i,n} \approx \frac{1+\sigma_i\gamma_{2,i}}{1+\sigma_i(\gamma_{1,i}+\gamma_{2,i})}\left(u_{1,i,n}-\frac{\sigma_i\gamma_{1,i}}{1+\sigma_i\gamma_{2,i}} u_{2,i,n}\right) \\
					\wt y_{i,n} \approx \frac{1}{1+\sigma_i\gamma_{2,i}}\left(u_{2,i,n} -\sigma_i\gamma_{2,i}\wt z_{i,n}\right) \\
					\wt v_{i,n} \approx v_{i,n} + \sigma_i\left(L_i(2\wt x_{n}-x_n)-r_i - \wt z_{i,n}- \wt y_{i,n}\right)
				\end{array} \right.\\
		x_{n+1} = x_n + \lambda_n (\wt x_{n}-x_n ) \\
		\text{For }i=1,\ldots,m  \\
				\left\lfloor \begin{array}{l}
					p_{i,n+1} = p_{i,n} + \lambda_n(\wt p_{i,n}-p_{i,n})\\
					q_{i,n+1} = q_{i,n} + \lambda_n(\wt q_{i,n}-q_{i,n})\\
					z_{i,n+1} = z_{i,n} + \lambda_n(\wt z_{i,n}-z_{i,n})\\
					y_{i,n+1} = y_{i,n} + \lambda_n(\wt y_{i,n}-y_{i,n})\\
					v_{i,n+1} = v_{i,n} + \lambda_n(\wt v_{i,n}-v_{i,n}).
				\end{array} \right. \\ \vspace{-4mm}
		\end{array}
		\right.
		\end{array}
	\end{align}}
\end{algorithm}

\begin{theorem}\label{ic_th1}
For Problem \ref{ic_problem}, suppose that
\begin{align}\label{ic_th1_inclusion}
	z \in \ran \bigg( A +  \sum_{i=1}^m L_i^*\Big(\big(K_i^*\circ B_i\circ K_i \big)\Box \big(M_i^*\circ D_i\circ M_i\big)\Big)(L_i \cdot -r_i) +C \bigg),
\end{align}
and consider the sequences generated by Algorithm \ref{ic_alg1}. Then there exists a primal-dual solution $(\bx, \f \bp, \f \bq, \f \by)$ to Problem \ref{ic_problem} such that
\begin{enumerate}[label={(\roman*)}]
	\setlength{\itemsep}{-2pt}
		\item  \label{ic_th1.01} $x_n \rightharpoonup \bx$, $p_{i,n} \rightharpoonup \bp_{i}$, $q_{i,n} \rightharpoonup \bq_{i}$ and $y_{i,n}\rightharpoonup \by_{i}$ for any $i=1,\ldots,m$ as $n \rightarrow +\infty$.
    \item  \label{ic_th1.02} if $C$ is uniformly monotone at $\bx$, then $x_n \rightarrow \bx$ as $n \rightarrow +\infty$.
	\end{enumerate}
\end{theorem}

\begin{proof}
We introduce the real Hilbert space $\fK=\h  \oplus \f \X \oplus \f \Y \oplus \f \g \oplus \f \g \oplus \f \g$  and let
\begin{align}
	\label{ic_hilbert2}
	\left\{\begin{array}{l}
		\f p = (p_{1},\ldots,p_{m}) \\
		\f q = (q_{1},\ldots,q_{m}) \\	
		\f y = (y_{1},\ldots,y_{m})
	\end{array}\right.
\text{and} \
\left\{\begin{array}{l}
		\f z = (z_{1},\ldots,z_{m}) \\
		\f v = (v_1,\ldots,v_m) \\
		\f r = (r_1,\ldots,r_m)
	\end{array}\right..
\end{align}
We introduce the maximally monotone operators
\begin{align*}
\f B : \f \X \rightarrow 2^{\f \X}, \ \f p \mapsto B_1p_{1} \times \ldots \times B_mp_{m} \ \mbox{and} \ \f D : \f \Y \rightarrow 2^{\f \Y},\ \f q \mapsto D_1q_1\times \ldots \times D_m q_m.
\end{align*}
Further, consider the set-valued operator
\begin{align*}
	\f M : \fK \rightarrow 2^{\fK}, \ (x,\f p,\f q, \f z,\f y, \f v) \! \mapsto \! (-z + Ax) \times \f B^{-1} \f p \times \f D^{-1} \f q \times (-\f v, -\f v, \f r + \f z+ \f y ),
\end{align*}
which is maximally monotone,  since $A$, $\f B$ and $\f D$ are maximally monotone (cf. \cite[Proposition 20.22 and Proposition 20.23]{BauCom11}) and the linear bounded operator
\begin{align*}
(x,\f p,\f q, \f y, \f z,\f v) \mapsto (0,\f 0, \f 0, -\f v, -\f v,\f z + \f y )
\end{align*}
is skew and hence maximally monotone (cf. \cite[Example 20.30]{BauCom11}). Therefore, $\f M$ can be written as the sum of two maximally monotone operators, one of them having full domain, fact which leads to the maximality of $\f M$ (see, for instance, \cite[Corollary 24.4(i)]{BauCom11}). Furthermore, consider the linear bounded operators
\begin{align*}
\wt K : \f \g \rightarrow \f \X,\ \f z \mapsto (K_1z_{1},\ldots,K_mz_{m}),\ \wt M : \f \g \rightarrow \f \Y,\ \f y \mapsto (M_1 y_{1}, \ldots,M_m y_{m}).
\end{align*}
and
\begin{align*}
	&\f S : \fK \rightarrow \fK, \\
	&(x,\f p,\f q, \f z, \f y,\f v) \mapsto \left(\sum_{i=1}^mL_i^*v_i,-\wt K \f z,-\wt M \f y,\wt K^* \f p,\wt M^* \f q,-L_1x, \ldots, -L_mx\right)
\end{align*}
The operator $\f S$ is skew, as well, hence maximally monotone. As $\dom \f S = \fK$, the sum $\f M + \f S$ is maximally monotone (see \cite[Corollary 24.4(i)]{BauCom11}).

Finally, we introduce the monotone operator
\begin{align*}
	\f Q:\fK \rightarrow \fK,\ (x,\f p,\f q, \f z, \f y ,\f v) \mapsto (Cx,\f 0, \f 0, \f 0, \f 0, \f 0)
\end{align*}
which is, obviously, $\mu^{-1}$-cocoercive. By making use of \eqref{ic_equivsol}, we observe that
{\allowdisplaybreaks
\begin{align*}
\eqref{ic_th1_inclusion} & \
	 \Leftrightarrow \exists\, (x, \f p, \f q, \f y) \in \h \oplus \f \X \oplus \f \Y \oplus \f \g: \ \left\{
		\begin{array}{l}
		z - \sum_{i=1}^mL_i^*K_i^* p_i \in Ax +Cx, \\
        K_i(L_ix - y_i-r_i) \in B_i^{-1} p_i,\ i=1,\ldots,m,\\
		M_i y_i \in D_i^{-1} q_i,\ i=1,\ldots,m,\\
        K_i^* p_i = M_i^* q_i, \ i=1,\ldots,m.
	\end{array}\right.\\
	& \Leftrightarrow \begin{array}{l}\exists\,(x,\f p, \f q) \in \h \oplus  \f \X \oplus \f \Y \\
\exists\,(\f z, \f y, \f v) \in \f \g \oplus \f \g \oplus \f \g \end{array} :
\left\{\begin{array}{l}0\in -z +Ax+\sum_{i=1}^mL_i^*v_i +Cx,  \\
0 \in -K_i z_{i} +B_i^{-1}p_{i},\ i=1,\ldots,m, \\
0 \in -M_iy_{i} + D_i^{-1}q_{i},\ i=1,\ldots,m, \\
0 = K_i^*p_{i} -v_i,\ i=1,\ldots,m, \\
0 = M_i^*q_{i}-v_i,\ i=1,\ldots,m,\\
0 =  r_i + z_{i}+y_{i} -L_ix,\ i=1,\ldots,m
\end{array}\right. \\
	& \Leftrightarrow \exists\,(x,\f p,\f q,\f z,\f y,\f v) \in \zer(\f M + \f S + \f Q).
\end{align*}}
From here it follows that
\begin{align}
	\label{ic_zer_opt_solution}
	&(\bx,\f \bp,\f \bq,\f \bz,\f \by,\f \bv) \in \zer(\f M + \f S+\f Q) \notag \\
	&\Rightarrow \left\{
		\begin{array}{l}
		z - \sum_{i=1}^mL_i^*K_i^*\bp_i \in A\bx +C\bx, \\
        K_i(L_i\bx -\by_i-r_i) \in B_i^{-1}\bp_i,\ i=1,\ldots,m,\\
		M_i\by_i \in D_i^{-1}\bq_i,\ i=1,\ldots,m,\\
        K_i^*\bp_i = M_i^*\bq_i,\ i=1,\ldots,m.
	\end{array}\right. \notag \\
	& \Leftrightarrow (\bx, \f \bp, \f \bq, \f \by) \text{ is a primal-dual solution to Problem \ref{ic_problem}}.
\end{align}
Further, for positive real values $\tau, \theta_{1,i},\theta_{2,i},\gamma_{1,i},\gamma_{2,i},\sigma_i \in \R_{++}$, $i=1,\ldots,m$, we introduce the notations
\begin{align*}
\left\{ \begin{array}{l}
	\frac{\f p}{\theta_1}=\left(\frac{p_{1}}{\theta_{1,1}},\ldots,\frac{p_{m}}{\theta_{1,m}}\right) \\
	\frac{\f q}{\theta_2}=\left(\frac{q_{1}}{\theta_{2,1}},\ldots,\frac{q_{m}}{\theta_{2,m}}\right)
 \end{array}\right.
\left\{ \begin{array}{l}
	\frac{\f z}{\gamma_1}=\left(\frac{z_{1}}{\gamma_{1,1}},\ldots,\frac{z_{m}}{\gamma_{1,m}}\right) \\
	\frac{\f y}{\gamma_2}=\left(\frac{y_{1}}{\gamma_{2,1}},\ldots,\frac{y_{m}}{\gamma_{2,m}}\right)
 \end{array}\right.
\left\{ \begin{array}{l}
	\frac{\f v}{\sigma}=\left(\frac{v_1}{\sigma_1},\ldots,\frac{v_m}{\sigma_m}\right)
 \end{array}\right.
\end{align*}
and define the linear bounded operator
\begin{align*}
	\f V : \fK \rightarrow \fK,\ &(x,\f p, \f q,\f z, \f y,\f v) \mapsto
	\left(\frac{x}{\tau},\frac{\f p}{\theta_1},\frac{\f q}{\theta_2},\frac{\f z}{\gamma_1},\frac{\f y}{\gamma_2},\frac{\f v}{\sigma} \right) \\
	&+ \left(-\sum_{i=1}^mL_i^*v_i,\wt K \f z,\wt M \f y,\wt K^* \f p,\wt M^* \f q,-L_1x,\ldots,-L_mx\right).
\end{align*}
It is a simple calculation to prove that $\f V$ is self-adjoint. Furthermore, the operator $\f V$ is $\rho$-strongly positive with
$$\rho = \left(1-\overline{\alpha} \right) \min_{i =1,\ldots,m}\bigg\{\frac{1}{\tau},\frac{1}{\theta_{1,i}},\frac{1}{\theta_{2,i}}, \frac{1}{\gamma_{1,i}},\frac{1}{\gamma_{2,i}},\frac{1}{\sigma_i} \bigg\} > 0,$$
for
$$ \overline{\alpha} = \max\left\{\sqrt{\tau \sum_{i=1}^m \sigma_i \|L_i\|^2}, \max_{j =1,\ldots,m}\left\{\sqrt{\theta_{1,j}\gamma_{1,j}\|K_j\|^2}, \sqrt{\theta_{2,j}\gamma_{2,j}\|M_j\|^2}\right\} \right\}. $$
The fact that $\rho$ is a positive real number follows by the assumptions made in Algorithm \ref{ic_alg1}. Indeed, using that $2ab \leq \alpha a^2 + \frac{b^2}{\alpha}$ for every $a,\,b\in \R$ and every $\alpha \in \R_{++}$, it yields for any $i=1,\ldots,m$
\begin{align}
	\label{ic_ineq1}
	\begin{aligned}
	2\| L_i \| \|x\|_{\h} \|v_i\|_{\g_i} &\leq  \frac{\sigma_i \|L_i\|^2}{\sqrt{\tau \sum_{i=1}^m \sigma_i \|L_i\|^2}} \|x\|_{\h}^2 + \frac{\sqrt{\tau \sum_{i=1}^m \sigma_i \|L_i\|^2}}{\sigma_i} \|v_i\|_{\g_i}^2,\\
	2\| K_i \| \|p_i\|_{\X_i} \|z_i\|_{\g_i} &\leq \frac{\gamma_{1,i} \|K_i\|}{\sqrt{\theta_{1,i} \gamma_{1,i}}} \|p_i\|_{\X_i}^2 + \frac{\sqrt{\theta_{1,i} \gamma_{1,i} \|K_i\|^2}}{\gamma_{1,i}} \|z_i\|_{\g_i}^2, \\
	2\| M_i \| \|q_i\|_{\Y_i} \|y_i\|_{\g_i} &\leq \frac{\gamma_{2,i} \|M_i\|}{\sqrt{\theta_{2,i} \gamma_{2,i}}} \|q_i\|_{\Y_i}^2 + \frac{\sqrt{\theta_{2,i} \gamma_{2,i} \|M_i\|^2}}{\gamma_{2,i}} \|y_i\|_{\g_i}^2. \\
	\end{aligned}
\end{align}
Consequently, for each $\fx=(x,\f p, \f q, \f z, \f y, \f v) \in \fK$, using the Cauchy-Schwarz inequality and \eqref{ic_ineq1}, it follows that
\begin{align}
	\<\fx, \f V\fx \>_{\fK}
	&= \frac{\|x\|_{\h}^2}{\tau} + \sum_{i=1}^m\left[\frac{\|p_i\|_{\X_i}^2}{\theta_{1,i}} + \frac{\|q_i\|_{\Y_i}^2}{\theta_{2,i}}+\frac{\|z_i\|_{\g_i}^2}{\gamma_{1,i}}+\frac{\|y_i\|_{\g_i}^2}{\gamma_{2,i}}+ \frac{\|v_i\|_{\g_i}^2}{\sigma_i}\right] \notag \\
	&\quad - 2\sum_{i=1}^m \< L_i x, v_i\>_{\g_i} + 2\sum_{i=1}^m \<p_i,K_iz_i\>_{\X_i} + 2\sum_{i=1}^m \<q_i,M_iy_i\>_{\Y_i} \notag \\
	&\geq  \left(1-\overline{\alpha} \right) \min_{i =1,\ldots,m}\bigg\{\frac{1}{\tau},\frac{1}{\theta_{1,i}},\frac{1}{\theta_{2,i}}, \frac{1}{\gamma_{1,i}},\frac{1}{\gamma_{2,i}},\frac{1}{\sigma_i} \bigg\} \|\f x\|_{\fK}^2 \notag \\
	&= \rho  \|\fx\|_{\fK}^2.
\end{align}
Since $\f V$ is $\rho$-strongly positive, we have $\cl(\ran \f V)=\ran \f V$ (cf. \cite[Fact 2.19]{BauCom11}), $\zer \f V = \{0\}$ and, as $(\ran \f V)^{\bot} = \zer \f V^* = \zer \f V = \{0\}$ (see, for instance, \cite[Fact 2.18]{BauCom11}), it holds $\ran \f V = \fK$. Consequently, $\f V^{-1}$ exists and $\|\f V^{-1}\|\leq \frac{1}{\rho}$.

In consideration of \eqref{ic_errors},  the algorithmic scheme \eqref{ic_A1} can equivalently be written in the form
	\begin{align}\label{ic_A1.1}
	  \left(\forall n\geq 0\right) \begin{array}{l} \left\lfloor \begin{array}{l}
		\frac{x_n- \wt x_{n}}{\tau} - \sum_{i=1}^m L_i^*( v_{i,n}- \wt v_{i,n}) -Cx_n \\ \hspace{3cm} \in -z + A(\wt x_{n}-a_n)  + \sum_{i=1}^mL_i^*\wt v_{i,n} -\frac{a_n}{\tau} \\
		\text{For }i=1,\ldots,m  \\
				\left\lfloor \begin{array}{l}
				 \frac{p_{i,n}-\wt p_{i,n}}{\theta_{1,i}} + K_i(z_{i,n}-\wt z_{i,n}) \in B_i^{-1}(\wt p_{i,n}-b_{i,n}) -K_i \wt z_{i,n} -\frac{b_{i,n}}{\theta_{1,i}} \\
				 \frac{q_{i,n}-\wt q_{i,n}}{\theta_{2,i}} + M_i(y_{i,n}-\wt y_{i,n}) \in D_i^{-1}(\wt q_{i,n}-d_{i,n}) -M_i \wt y_{i,n} -\frac{d_{i,n}}{\theta_{2,i}} \\
				 \frac{z_{i,n}-\wt z_{i,n}}{\gamma_{1,i}} + K_i^*(p_{i,n}-\wt p_{i,n}) = -\wt v_{i,n} +K_i^*\wt p_{i,n}-e_{1,i,n} \\
				 \frac{y_{i,n}-\wt y_{i,n}}{\gamma_{2,i}} + M_i^*(q_{i,n}-\wt q_{i,n}) = -\wt v_{i,n} +M_i^*\wt q_{i,n}-e_{2,i,n} \\
				 \frac{v_{i,n} - \wt v_{i,n}}{\sigma_i}	- L_i (x_n - \wt x_n) = r_i  + \wt z_{i,n}+ \wt y_{i,n} - L_i \wt x_{n} -e_{3,i,n}\\
				\end{array} \right.\\
		\f x_{n+1} = \f x_n + \lambda_n (\f{\wt x_n} - \f x_n ),	
		\end{array}
		\right.
		\end{array}
	\end{align}
where
\begin{align*}
\left\{
	\begin{array}{l}
	\f p_{n} = (p_{1,n},\ldots p_{m,n}) \in \f \X \\
	\f q_{n} =(q_{1,n},\ldots,q_{m,n}) \in \f \Y \\
	\f z_{n} =(z_{1,n},\ldots,z_{m,n}) \in \f \g \\
	\f y_{n} =(y_{1,n},\ldots,y_{m,n}) \in \f \g \\
	\f v_{n} =(v_{1,n},\ldots,v_{m,n}) \in \f \g
	\end{array}\right.\quad
\left\{
	\begin{array}{l}
	\f{\wt p_{n}} = (\wt p_{1,n},\ldots \wt p_{m,n}) \in \f \X \\
	\f{\wt q_{n}} =(\wt q_{1,n},\ldots,\wt q_{m,n}) \in \f \Y \\
	\f{\wt z_{n}} =(\wt z_{1,n},\ldots,\wt z_{m,n}) \in \f \g \\
	\f{\wt y_{n}} =(\wt y_{1,n},\ldots,\wt y_{m,n}) \in \f \g \\
	\f{\wt v_{n}} =(\wt v_{1,n},\ldots,\wt v_{m,n}) \in \f \g
	\end{array}\right.
\end{align*}
\begin{align*}
\left\{
	\begin{array}{l}
	\f x_n = (x_n,\f p_{n},\f q_{n}, \f z_{n}, \f y_{n} , \f v_{n}) \in \fK \\
	\f{\wt x_n} = (\wt x_n, \f{\wt p_{n}}, \f{\wt q_{n}}, \f{\wt z_{n}}, \f{\wt y_{n}}, \f{\wt v_{n}}) \in \fK.
	\end{array}\right.
\end{align*}
Also, for any $n \geq 0$, we consider sequences defined  by
\begin{align}
\label{ic_error_sequences}
\left\{
	\begin{array}{l}
	a_n \in \h \\
	\f b_{n} = (b_{1,n},\ldots b_{m,n}) \in \f \X \\
	\f d_{n} =(d_{1,n},\ldots,d_{m,n}) \in \f \Y
	\end{array}\right.
\text{and} \
\left\{
	\begin{array}{l}
	\f e_{1,n} =(e_{1,1,n},\ldots,e_{1,m,n}) \in \f \g \\
	\f e_{2,n} =(e_{2,1,n},\ldots,e_{2,m,n}) \in \f \g, \\
	\f e_{3,n} =(e_{3,1,n},\ldots,e_{3,m,n}) \in \f \g
	\end{array}\right.
\end{align}
that are summable in the corresponding norm. Further, by denoting for any $n \geq 0$
\begin{align*}
\left\{
	\begin{array}{l}
	\f e_n = (a_n,\f b_{n},\f d_{n}, \f 0, \f 0 , \f 0) \in \fK \\
	\f e^{\tau}_n = \left(\frac{a_n}{\tau},\frac{\f b_{n}}{\theta_1},\frac{\f d_n}{\theta_2},\f e_{1,n},\f e_{2,n}, \f e_{3,n}\right) \in \fK,
	\end{array}\right.
\end{align*}
which are also terms of summable sequences in the corresponding norm, it yields that the scheme in \eqref{ic_A1.1} is equivalent to
\begin{align}
	\label{ic_A1.2}
	\left(\forall n\geq 0\right)  \left\lfloor \begin{array}{l}
	\f V(\f x_n - \f{\wt x_n}) -\f Q \f x_n \in \left(\f M +\f S\right) (\f{\wt x_n} -\f e_n) + \f S \f e_n - \f e_n^{\tau}  \\
	\f x_{n+1} = \f x_n + \lambda_n \left(\f{\wt x_n}-\f x_n\right).	
	\end{array}
	\right.
\end{align}
We now introduce the notations
\begin{align}\label{ic_def1.1}
		\f A_{\fK} := \f V^{-1}\left(\f M +\f S\right) \ \mbox{and} \ \f B_{\fK}:=\f V^{-1}\f Q
\end{align}
and the summable sequence with terms $\f e^{\f V}_n=\f V^{-1}\left((\f V + \f S) \f e_n - \f e_n^{\tau}\right)$ for any $n \geq 0$. Then, for any $n \geq 0$, we have
\begin{align}\label{ic_inc1.1}
	& \f V(\f x_n - \f{\wt x_n}) -\f Q \f x_n \in \left(\f M +\f S\right) (\f{\wt x_n} -\f e_n) + \f S \f e_n - \f e_n^{\tau} \notag \\
	\Leftrightarrow \ & \f V\f x_n - \f Q \f x_n \in \left(\f V + \f M +\f S\right)(\f{\wt x_n} -\f e_n) + (\f V + \f S) \f e_n - \f e_n^{\tau} \notag \\
	\Leftrightarrow \ & \f x_n - \f V^{-1} \f Q \f x_n \in \left(\id + \f V^{-1}\left(\f M +\f S\right)\right)(\f{\wt x_n} -\f e_n) + \f V^{-1}\left((\f V + \f S) \f e_n - \f e_n^{\tau} \right)  \notag \\
	\Leftrightarrow \ & \f{\wt x_n} = \left(\id + \f V^{-1}\left(\f M +\f S\right)\right)^{-1}\left(\fx_n - \f V^{-1} \f Q \f x_n - \f e^{\f V}_n \right) + \f e_n \notag \\
  \Leftrightarrow \ & \f{\wt x_n} = \left(\id + \f A_{\fK}\right)^{-1}\left(\fx_n - \f B_{\fK}\f x_n - \f e^{\f V}_n\right) + \f e_n.
\end{align}
Taking into account that the resolvent is Lipschitz continuous, the sequence having as terms
$$ \f e_n^{\f A_{\fK}} = J_{\f A_{\fK}} \left(\fx_n - \f B_{\fK}\f x_n - \f e^{\f V}_n\right) - J_{\f A_{\fK}} \left(\fx_n - \f B_{\fK}\f x_n\right) + \f e_n \ \forall n \geq 0$$
is summable and we have
$$\f{\wt x_n} = J_{\f A_{\fK}} \left(\fx_n - \f B_{\fK}\f x_n \right) +\f e_n^{\f A_{\fK}} \ \forall n \geq 0.$$
Thus, the iterative scheme in \eqref{ic_A1.2} becomes
\begin{align}
	\label{ic_A1.3}
	\left(\forall n\geq 0\right)  \left\lfloor \begin{array}{l}
	\f{\wt x_n} \approx J_{\f A_{\fK}}\left(\f x_n - \f B_{\fK} \f x_n\right)  \\
	\f x_{n+1} = \f x_n + \lambda_n(\f{\wt x_n}- \f x_n),	
	\end{array}
	\right.
\end{align}
which shows that the algorithm we propose in this subsection has the structure of a forward-backward method.

In addition, let us observe that
\begin{align*}
	\zer \left( \f A_{\fK} + \f B_{\fK}\right) = \zer \left( \f V^{-1}\left( \f M + \f S + \f Q\right) \right) = \zer \left( \f M + \f S + \f Q \right).
\end{align*}
We then introduce the Hilbert space $\fK_{\f V}$ with inner product and norm respectively defined, for $\fx,\fy \in \fK$, via
\begin{align}\label{ic_HSKV}
 \< \fx,\fy \>_{\fK_{\f V}} = \< \fx, \f V \fy \>_{\fK} \text{ and } \|\fx\|_{\fK_{\f V}} = \sqrt{\< \fx, \f V \fx \>_{\fK}}.
\end{align}
Since $\f M+ \f S$ and $\f Q$ are maximally monotone on $\fK$, the operators $\f A_{\fK}$ and $\f B_{\fK}$ are maximally monotone on $\fK_{\f V}$. Moreover, since $\f V$ is self-adjoint and $\rho$-strongly positive, one can easily see that weak and strong convergence in $\fK_{\f V}$ are equivalent with weak and strong convergence in $\fK$, respectively. By making use of $\|V^{-1}\|\leq \frac{1}{\rho}$, one can show that $\f B_{\fK}$ is $(\mu^{-1}\rho)$-cocoercive on $\fK_{\f V}$. Indeed, we get for $\f x,\,\f y \in \fK_{\f V}$ that (see, also, \cite[Eq. (3.35)]{Vu13})
\begin{align}
	\label{ic_cocoercive}
	\< \f x - \f y, \f B_{\fK}\f x - \f B_{\fK}\f y \>_{\fK_{\f V}}
	&= \< \f x-\f y, \f Q \f x-\f Q \f y\>_{\fK} \notag \\
	&\geq \mu^{-1} \|\f Q \f x - \f Q \f y\|_{\fK}^2 \notag \\
	&\geq \mu^{-1} \|\f V^{-1}\|^{-1} \|\f V^{-1} \f Q \f x - \f V^{-1} \f Q \f y\|_{\fK} \|\f Q \f x - \f Q \f y\|_{\fK} \notag \\
	&\geq \mu^{-1} \|\f V^{-1}\|^{-1} \< \f B_{\fK} \f x - \f B_{\fK}\f y, \f Q \f x - \f Q \f y \>_{\fK} \notag \\
	&= \mu^{-1} \|\f V^{-1}\|^{-1} \| \f B_{\fK} \f x - \f B_{\fK} \f y\|_{\fK_{\f V}}^2 \notag \\
	&\geq \mu^{-1} \rho  \| \f B_{\fK} \f x - \f B_{\fK} \f y\|_{\fK_{\f V}}^2.
\end{align}
As our assumption imposes that $2\mu^{-1}\rho>1$, we can use the statements given in \cite[Corollary 6.5]{Com04} in the context of an error tolerant forward-backward algorithm in order to establish the desired convergence results.

\ref{ic_th1.01} By Corollary 6.5 in \cite{Com04}, the sequence $(\f x_n)_{n \geq 0}$ converges weakly in $\fK_{\f V}$ (and therefore in $\fK$) to some $\f \bx = (\bx,\f \bp,\f \bq,\f \bz,\f \by,\f \bv) \in \zer \left( \f A_{\fK} + \f B_{\fK}\right) = \zer \left(\f M + \f S+ \f Q\right)$. By \eqref{ic_zer_opt_solution}, it thus follows that $(\bx, \f \bp, \f \bq, \f \by)$ is a primal-dual solution with respect to Problem \ref{ic_problem}.

\ref{ic_th1.02} From \cite{Com04} it follows
$$ \sum_{n \geq 0} \|\f B_{\fK} \f x_n - \f B_{\fK} \f \bx\|^2_{\fK_{\f V}} < + \infty, $$
and therefore we have $\f B_{\fK} \f x_n \rightarrow \f B_{\fK} \f \bx_n$ or, equivalently, $\f Q \f x_n \rightarrow \f Q \f \bx$ as $n \rightarrow +\infty$. Considering the definition of $\f Q$, one can see that this implies $C x_n \rightarrow C \bx$ as $n \rightarrow +\infty$. As $C$ is uniformly monotone, there exists an increasing function $\phi_C : \left[0, +\infty\right) \rightarrow \left[0,+\infty\right]$ vanishing only at $0$ such that
$$\phi_C(\|x_n-\bx\|) \leq \< x_n - \bx, Cx_n - C\bx \> \leq \|x_n-\bx\| \|Cx_n - C\bx\| \ \forall n \geq 0.$$
The boundedness of $(x_n-\bx)_{n\geq 0}$ and the convergence $Cx_n \rightarrow C \bx$ further imply that $x_n \rightarrow \bx$ as $n \rightarrow +\infty$.
\end{proof}

\begin{remark}
	\label{ic_remark_C0}
	Suppose that $C: \h \rightarrow \h$, $x\mapsto \{0\}$ in Problem \ref{ic_problem}. Then condition \eqref{ic_condition_step_length} simplifies to
	\begin{align*}
		\max\bigg\{\tau \sum_{i=1}^m \sigma_i \|L_i\|^2, \max_{j=1,\ldots,m}\left\{\theta_{1,j}\gamma_{1,j}\|K_j\|^2, \theta_{2,j}\gamma_{2,j}\|M_j\|^2\right\} \bigg\}<1.
	\end{align*}
Then the scheme \eqref{ic_A1.3} reads
	\begin{align}
		\label{ic_A1.10}
		\left(\forall n\geq 0\right)  \left\lfloor \begin{array}{l}
		\f x_{n+1} \approx \f x_n + \lambda_n(J_{\f A_{\fK}}\f x_n- \f x_n),
		\end{array}
		\right.
	\end{align}
and it can be shown to convergence under the relaxed assumption that $(\lambda_n)_{n\geq 0} \subseteq \left[\varepsilon,2-\varepsilon\right]$, for $\varepsilon \in (0,1)$ (see, for instance, \cite{Com01,Com04,EckBer92}).
\end{remark}

\begin{remark}
	\label{ic_remark_implementation_FB}
\begin{enumerate}[label={(\roman*)}]
	\setlength{\itemsep}{-2pt}
		\item \label{ic_rm1.01} When implementing Algorithm \ref{ic_alg1}, the term $L_i(2\wt x_{n}-x_n)$ should be stored in a separate variable for any $i =1,\ldots,m$. Taking this into account, each linear bounded operator occurring in Problem \ref{ic_problem} needs to be processed once via some forward evaluation and once via its adjoint.
    \item \label{ic_rm1.02} The maximally monotone operators $A$, $B_i$ and $D_i,\ i=1,\ldots,m,$ in Problem \ref{ic_problem} are accessed via their resolvents (so-called backward steps), also by taking into account the relation between the resolvent of a maximally monotone operator and its inverse given in \eqref{ic_res-identity}.
	\item \label{ic_rm1.04} The possibility of performing a forward step for the cocoercive monotone operator $C$ is an important aspect, since forward steps are usually much easier to implement than resolvents (resp. proximity operators). Due to the Baillon-Haddad theorem (cf. \cite[Corollary 18.16]{BauCom11}), each $\mu$-Lipschitzian gradient with $\mu \in \R_{++}$ of a convex and Fr\'echet differentiable function $f:\h \rightarrow \R$ is $\mu^{-1}$-cocoercive.
	\end{enumerate}	
\end{remark}

\subsection{An algorithm of forward-backward-forward type}\label{ic_subsecFBF}
In this subsection we propose a forward-backward-forward type algorithm for solving Problem \ref{ic_problem}, with the modification that the operator $C:\h \rightarrow \h$ is assumed to be $\mu$-Lipschitz continuous 
for some $\mu \in \R_{++}$, but not necessarily $\mu^{-1}$-cocoercive.

\begin{algorithm}\label{ic_alg2} \text{ }\newline
Let $x_0 \in \h$, and for any $i=1,...,m$, let $p_{i,0}\in\X_i$, $q_{i,0} \in \Y_i$, and $z_{i,0},\ y_{i,0},\ v_{i,0} \in \g_i$. Set
\begin{align}
	\label{ic_condition_beta}
	\beta = \mu + \sqrt{\max \bigg\{ \sum_{i=1}^m\|L_i\|^2, \max_{j=1,...,m}\big\{ \|K_j\|^2, \|M_j\|^2 \big\}  \bigg\}},
\end{align}
 let $\varepsilon \in \left (0,\frac{1}{\beta+1}\right)$, $(\gamma_n)_{n\geq 0}$ a sequence in $\left[\varepsilon,\frac{1-\varepsilon}{\beta}\right]$ and set
	\begin{align}\label{ic_A1_2}
	  \left(\forall n\geq 0\right) \begin{array}{l} \left\lfloor \begin{array}{l}
		\wt x_n \approx J_{\gamma_n A}\left( x_n - \gamma_n\left(Cx_n + \sum_{i=1}^m L_i^* v_{i,n} - z\right) \right) \\
		\text{For }i=1,\ldots,m  \\
				\left\lfloor \begin{array}{l}
					\wt p_{i,n} \approx J_{\gamma_n B_i^{-1}}\left(p_{i,n} +\gamma_n K_i z_{i,n}  \right) \\
					\wt q_{i,n} \approx J_{\gamma_n D_i^{-1}}\left(q_{i,n} +\gamma_n M_i y_{i,n}  \right) \\
					u_{1,i,n} \approx z_{i,n} -\gamma_n\left(K_i^*p_{i,n}-v_{i,n}-\gamma_n\left(L_ix_n-r_i\right)\right) \\
					u_{2,i,n} \approx y_{i,n} -\gamma_n\left(M_i^*q_{i,n}-v_{i,n}-\gamma_n\left(L_ix_n-r_i\right)\right) \\
					\wt z_{i,n} \approx \frac{1+\gamma_n^2}{1+2\gamma_n^2}\left(u_{1,i,n}-\frac{\gamma_n^2}{1+\gamma_n^2} u_{2,i,n}\right) \\
					\wt y_{i,n} \approx \frac{1}{1+\gamma_n^2}\left(u_{2,i,n} -\gamma_n^2 \wt z_{i,n}\right) \\
					\wt v_{i,n} \approx v_{i,n} + \gamma_n\left(L_ix_n-r_i - \wt z_{i,n}- \wt y_{i,n}\right) 
				\end{array} \right.\\
		x_{n+1} \approx \wt x_n + \gamma_n(Cx_n-C \wt x_n + \sum_{i=1}^m L_i^*(v_{i,n}-\wt v_{i,n})) \\
		\text{For }i=1,\ldots,m  \\
				\left\lfloor \begin{array}{l}
					p_{i,n+1} \approx \wt p_{i,n} - \gamma_n(K_i(z_{i,n}-\wt z_{i,n}))\\
					q_{i,n+1} \approx \wt q_{i,n} - \gamma_n(M_i(y_{i,n}-\wt y_{i,n}))\\
					z_{i,n+1} \approx \wt z_{i,n} + \gamma_n(K_i^*(p_{i,n}-\wt p_{i,n}))\\
					y_{i,n+1} \approx \wt y_{i,n} + \gamma_n(M_i^*(q_{i,n}-\wt q_{i,n}))\\
					v_{i,n+1} \approx \wt v_{i,n} - \gamma_n(L_i(x_n- \wt x_n)).
				\end{array} \right. \\ \vspace{-4mm}
		\end{array}
		\right.
		\end{array}
	\end{align}
\end{algorithm}

\begin{theorem}\label{ic_th2}
In Problem \ref{ic_problem}, let $C:\h \rightarrow \h$ be $\mu$-Lipschitz continuous for $\mu\in\R_{++}$, suppose that
\begin{align}\label{ic_th2_inclusion}
	z \in \ran \bigg( A +  \sum_{i=1}^m L_i^*\Big(\big(K_i^*\circ B_i\circ K_i \big)\Box \big(M_i^*\circ D_i\circ M_i\big)\Big)(L_i \cdot -r_i) +C \bigg),
\end{align}
and consider the sequences generated by Algorithm \ref{ic_alg2}. Then there exists a primal-dual solution $(\bx, \f \bp, \f \bq, \f \by)$ to Problem \ref{ic_problem} such that 
\begin{enumerate}[label={(\roman*)}]
	\setlength{\itemsep}{-2pt}
		\item  \label{ic_th2.01} $\sum_{n\geq 0}\|x_n- \wt x_n\|^2 < +\infty$ and for any $i=1,...,m$
		$$\sum_{n\geq 0} \|p_{i,n} - \wt p_{i,n}\|^2 < +\infty, \ \sum_{n\geq 0} \|q_{i,n} - \wt q_{i,n}\|^2 < +\infty   \text{ and }\sum_{n\geq 0} \|y_{i,n} - \wt y_{i,n}\|^2 < +\infty.$$
		\item  \label{ic_th2.02} $x_n \rightharpoonup \bx$, $\wt x_n \rightharpoonup \bx$, and for any $i=1,...,m$
		\begin{align*}
			 \left\{\begin{array}{l} p_{i,n}\rightharpoonup \bp_{i,n} \\ \wt p_{i,n}\rightharpoonup \bp_{i,n} \end{array}\right., \
		   \left\{\begin{array}{l} q_{i,n}\rightharpoonup \bq_{i,n} \\ \wt q_{i,n}\rightharpoonup \bq_{i,n} \end{array}\right. \text{ and }
			 \left\{\begin{array}{l} y_{i,n}\rightharpoonup \by_{i,n} \\ \wt y_{i,n}\rightharpoonup \by_{i,n} \end{array}\right.\!\!.
		\end{align*}
	\end{enumerate}
\end{theorem}
\begin{proof}
	As in the proof of Theorem \ref{ic_th1}, consider $\fK=\h  \oplus \f \X \oplus \f \Y \oplus \f \g \oplus \f \g \oplus \f \g$ along with the notations introduced in \eqref{ic_hilbert2}. Further, let the operators
	$\f M:\fK \rightarrow 2^{\fK}$, $\f S:\fK \rightarrow \fK$ and $\f Q:\fK \rightarrow \fK$ be defined as in the proof of the same result. 
	The operator $\f S + \f Q$ is monotone, Lipschitz continuous, hence maximally monotone (cf. \cite[Corollary 20.25]{BauCom11}), and it fulfills $\dom (\f S+\f Q)=\fK$. 
	Therefore the sum $\f M + \f S + \f Q$ is maximally monotone as well (see \cite[Corollary 24.4(i)]{BauCom11}). 
	
In the following we derive the Lipschitz constant of $\f S + \f Q$. For arbitrary
$$ 	\f x = (x,\f p,\f q, \f z, \f y , \f v) \ \text{ and } \ \f{\wt x} = (\wt x,\f{\wt p},\f{\wt q},\f{\wt z},\f{\wt y},\f{\wt v}) \in \fK, $$
by using the Cauchy-Schwarz inequality it yields,
{\allowdisplaybreaks
\begin{align}
	\label{ic_Lipschitz_constant}
	&\|(\f S + \f Q)\f x  - (\f S + \f Q)\f{\wt x} \| \leq \|\f Q \f x  - \f Q\f{\wt x} \| + \|\f S \f x  - \f S\f{\wt x} \| \notag \\
	&\leq \mu \|x- \wt x \| + \Bigg\| \bigg(\sum_{i=1}^mL_i^*(v_i-\wt v_i),-\wt K (\f z -\f{\wt z}),-\wt M (\f y -\f{\wt y}),\wt K^* (\f p- \f{\wt p}), \notag\\
	& \hspace{3cm}\wt M^* (\f q- \f{\wt q}),-L_1(x-\wt x),\ldots,-L_m(x-\wt x) \bigg)\Bigg\| \notag\\
	&= \mu \|x-\wt x \| + \bigg( \Big\|\sum_{i=1}^mL_i^*(v_i-\wt v_i)\Big\|^2 + \sum_{i=1}^m \Big[\|K_i(z_i-\wt z_i)\|^2 + \|M_i(y_i-\wt y_i)\|^2 \notag\\
	& \hspace{2.7cm} +\|K_i^*(p_i-\wt p_i)\|^2 + \|M_i^*(q_i-\wt q_i)\|^2 + \|L_i(x-\wt x)\|^2 \Big] \bigg)^{\frac{1}{2}} \notag\\
	&\leq \mu \|x-\wt x \| + \bigg( \Big( \sum_{i=1}^m \|L_i\|^2\Big)\Big( \|x-\wt x\|^2 + \sum_{i=1}^m \|v_i-\wt v_i\|^2 \Big) + \sum_{i=1}^m \Big[\|K_i\|^2 \|z_i-\wt z_i\|^2 \notag \\
	& \hspace{2.7cm} +\|M_i\|^2\|y_i-\wt y_i\|^2+ \|K_i\|^2 \|p_i-\wt p_i\|^2 + \|M_i\|^2\|q_i-\wt q_i\|^2 \Big] \bigg)^{\frac{1}{2}} \notag\\
	&\leq \left( \mu + \sqrt{\max \bigg\{ \sum_{i=1}^m\|L_i\|^2, \max_{j=1,...,m}\big\{ \|K_j\|^2, \|M_j\|^2 \big\}  \bigg\}} \right) \|\f x - \f{\wt x}\|.
\end{align}}
In the following we use the sequences in \eqref{ic_error_sequences} for modeling summable errors in the implementation. In addition we consider the summable sequences in $\fK$ with terms defined for any $n \geq 0$ as
$$ \f e_n=(a_n, \f b_n, \f d_n, \f 0, \f 0, \f 0) \text{ and } \f{\wt e}_n = (\f 0, \f 0,\f 0, \f e_{1,n}, \f e_{2,n}, \f e_{3,n}).$$
Note that \eqref{ic_A1_2} can equivalently be written as
\begin{align}\label{ic_A1.1_2}
	\left(\forall n\geq 0\right) \begin{array}{l} \left\lfloor \begin{array}{l}
		x_n - \gamma_n\big( Cx_n + \sum_{i=1}^m L_i^* v_{i,n}\big) \in \big( \id +\gamma_n(-z +A)\big)(\wt x_n-a_n) \\
		\text{For }i=1,\ldots,m  \\
				\left\lfloor \begin{array}{l}
					p_{i,n} + \gamma_n K_iz_{i,n} \in \big( \id + \gamma_n B_i^{-1} \big)(\wt p_{i,n}-b_{i,n}) \\
					q_{i,n} + \gamma_n M_iy_{i,n} \in \big( \id + \gamma_n D_i^{-1} \big)(\wt q_{i,n}-d_{i,n}) \\
					z_{i,n} - \gamma_nK_i^*p_{i,n} = \wt z_{i,n} - \gamma_n \wt v_{i,n} - e_{1,i,n} \\
					y_{i,n} - \gamma_nM_i^*q_{i,n} = \wt y_{i,n} - \gamma_n \wt v_{i,n} - e_{2,i,n} \\
					v_{i,n} + \gamma_n L_i x_n = \wt v_{i,n} + \gamma_n(r_i + \wt z_{i,n}+ \wt y_{i,n}) -e_{3,i,n}
				\end{array} \right.\\
		x_{n+1} \approx \wt x_n + \gamma_n(Cx_n-C \wt x_n + \sum_{i=1}^m L_i^*(v_{i,n}-\wt v_{i,n})) \\
		\text{For }i=1,\ldots,m  \\
				\left\lfloor \begin{array}{l}
					p_{i,n+1} \approx \wt p_{i,n} - \gamma_n(K_i(z_{i,n}-\wt z_{i,n}))\\
					q_{i,n+1} \approx \wt q_{i,n} - \gamma_n(M_i(y_{i,n}-\wt y_{i,n}))\\
					z_{i,n+1} \approx \wt z_{i,n} + \gamma_n(K_i^*(p_{i,n}-\wt p_{i,n}))\\
					y_{i,n+1} \approx \wt y_{i,n} + \gamma_n(M_i^*(q_{i,n}-\wt q_{i,n}))\\
					v_{i,n+1} \approx \wt v_{i,n} - \gamma_n(L_i(x_n- \wt x_n)).
				\end{array} \right. \\ \vspace{-4mm}
		\end{array}
		\right.
		\end{array}
\end{align}
Therefore, \eqref{ic_A1.1_2} is nothing else than
\begin{align}
	\label{ic_A1.2_2}
	\left(\forall n\geq 0\right)  \left\lfloor \begin{array}{l}
	\f x_n - \gamma_n (\f S + \f Q) \f x_n  \in \left(\id + \gamma_n\f M \right)(\f{\wt x_n} -\f e_n) -\f{\wt e}_n \\
	\f x_{n+1} \approx \f{\wt x_n} + \gamma_n \left((\f S + \f Q) \f x_n - (\f S + \f Q) \f p_n\right).	
	\end{array}
	\right.
\end{align}
We now introduce the notations
\begin{align}\label{ic_def2.1}
		\f A_{\fK} := \f M \ \mbox{and} \ \f B_{\fK}:=\f S + \f Q.
\end{align}
Then \eqref{ic_A1.2_2} is
\begin{align}
	\label{ic_A1.3_2}
	\left(\forall n\geq 0\right)  \left\lfloor \begin{array}{l}
	\f{\wt x_n} = J_{\gamma_n \f A_{\fK}}\left(\f x_n - \gamma_n \f B_{\fK} \f x_n + \f{\wt e}_n \right) + \f e_n \\
	\f x_{n+1} \approx \f{\wt x_n} + \gamma_n \left(\f B_{\fK} \f x_n - \f B_{\fK} \f{\wt x_n} \right).	
	\end{array}
	\right.
\end{align}
We observe that for 
$$\f e^{\fK}_n:=J_{\gamma_n \f A_{\fK}}\left(\f x_n - \gamma_n \f B_{\fK} \f x_n + \f{\wt e}_n \right) -J_{\gamma_n\f A_{\fK}}\left(\f x_n - \gamma_n \f B_{\fK} \f x_n \right) + \f e_n,$$
one has $\f{\wt x_n} = J_{\gamma_n \f A_{\fK}}\left(\f x_n - \gamma_n \f B_{\fK} \f x_n \right) + \f e^{\fK}_n$ for any $n \geq 0$ and it holds 
\begin{align*}
	\sum_{n\geq 0} \|\f e^{\fK}_n\| &= \sum_{n\geq 0} \|J_{\gamma_n \f A_{\fK}}\left(\f x_n - \gamma_n \f B_{\fK} \f x_n + \f{\wt e}_n \right) -J_{\gamma_n\f A_{\fK}}\left(\f x_n - \gamma_n \f B_{\fK} \f x_n \right) + \f e_n\| \\
	&\leq \sum_{n\geq 0} \left[ \|J_{\gamma_n\f A_{\fK}}\left(\f x_n - \gamma_n \f B_{\fK} \f x_n + \f{\wt e}_n \right) -J_{\gamma_n\f A_{\fK}}\left(\f x_n - \gamma_n \f B_{\fK} \f x_n \right)\| + \| \f e_n\| \right] \\
	&\leq \sum_{n\geq 0} \left[ \| \f{\wt e}_n\| + \| \f e_n\| \right] < +\infty.
\end{align*}
Thus, \eqref{ic_A1.3_2} becomes
\begin{align}
	\label{ic_A1.4_2}
	\left(\forall n\geq 0\right)  \left\lfloor \begin{array}{l}
	\f{\wt x_n} \approx J_{\gamma_n\f A_{\fK}}\left(\f x_n - \gamma_n \f B_{\fK} \f x_n \right) \\
	\f x_{n+1} \approx \f{\wt x_n} + \gamma_n \left(\f B_{\fK} \f x_n - \f B_{\fK} \f{\wt x_n}\right),	
	\end{array}
	\right.
\end{align} 
which is an error-tolerant forward-backward-forward method in $\fK$ whose convergence has been investigated in \cite{BriCom11}. Note that the exact version of this algorithm was proposed by Tseng in \cite{Tse00}.

\ref{ic_th2.01} By \cite[Theorem 2.5(i)]{BriCom11} we have $$ \sum_{n\geq 0} \|\f x_n - \f{\wt x_n}\|^2 < + \infty, $$ which yields $\sum_{n\geq 0} \|x_n -\wt x_n\|^2 < +\infty$ and for any $i=1,...,m$,
\begin{align*}
	 \sum_{n\geq 0} \|p_{i,n} - \wt p_{i,n}\|^2 < +\infty, \ \sum_{n\geq 0} \|q_{i,n} - \wt q_{i,n}\|^2 < +\infty   \text{ and }\sum_{n\geq 0} \|y_{i,n} - \wt y_{i,n}\|^2 < +\infty.
\end{align*}

\ref{ic_th2.02} Let $\f \bx=(\bx,\f \bp, \f \bq, \f \bz, \f \by, \f \bv) \in \zer (\f M + \f S + \f Q)$. Using \cite[Theorem 2.5(ii)]{BriCom11}, we obtain $\f x_n \rightharpoonup \f \bx$ and $\f{\wt x_n} \rightharpoonup \f \bx$. 
In consideration of \eqref{ic_zer_opt_solution}, it follows that $(\bx,\f \bp, \f \bq,\f \bv)$ is a primal-dual solution to Problem \ref{ic_problem}, $x_n \rightharpoonup \bx$, $\wt x_n \rightharpoonup \bx$, and for $i=1,...,m$
		\begin{align*}
			 \left\{\begin{array}{l} p_{i,n}\rightharpoonup \bp_{i,n} \\ \wt p_{i,n}\rightharpoonup \bp_{i,n} \end{array}\right., \
		   \left\{\begin{array}{l} q_{i,n}\rightharpoonup \bq_{i,n} \\ \wt q_{i,n}\rightharpoonup \bq_{i,n} \end{array}\right., \text{ and }
			 \left\{\begin{array}{l} y_{i,n}\rightharpoonup \by_{i,n} \\ \wt y_{i,n}\rightharpoonup \by_{i,n} \end{array}\right.\!\!.
		\end{align*}
\end{proof}

\begin{remark}
	\label{ic_remark_implementation_FBF}
\begin{enumerate}[label={(\roman*)}]
	\setlength{\itemsep}{-2pt}
		\item \label{ic_rm2.01} In contrast to Algorithm \ref{ic_alg1}, the iterative scheme in Algorithm \ref{ic_alg2} requires twice the amount of forward steps and is therefore more time-intensive. On the other hand, many steps in Algorithm \ref{ic_alg2} can be processed in parallel.
		\item \label{ic_rm2.02}  A related monotone inclusion problem involving linearly composed parallel sums of maximally monotone operators was investigated in \cite{BecCom13}, by proposing an iterative scheme which can be also reduced
		to a forward-backward-forward type iterative scheme. However, the algorithm there is different to the one given in Algorithm \ref{ic_alg2}.
	\end{enumerate}	
\end{remark}

\section{Application to convex minimization}\label{ic_secConvexMinimization}

In this section we employ the algorithms introduced in the previous one in the context of solving primal-dual pairs of convex optimization problems. The problem under consideration is as follows.

\begin{problem}\label{ic_problem_convex}
Let $\h$ be a real Hilbert space, $z \in \h$ and $f,\ h\in \Gamma(\h)$ such that $h$ is differentiable with $\mu$-Lipschitzian gradient for $\mu\in\R_{++}$.  Furthermore, for every $i=1,\ldots,m$, let $\g_i$, $\X_i$, $\Y_i$ be real Hilbert spaces, $r_i\in \g_i$, let $g_i \in \Gamma(\X_i)$ and $l_i \in \Gamma(\Y_i)$ and consider the nonzero linear bounded operators $L_i:\h \rightarrow \g_i$, $K_i:\g_i \rightarrow \X_i$ and $M_i:\g_i \rightarrow \Y_i$. Then we solve the primal optimization problem
\begin{align}
	\label{ic_p_primal_convex}
	\inf_{x\in \h}\bigg\{ f(x) + \sum_{i=1}^m \Big( \big(g_i \circ K_i\big) \Box \big(l_i \circ M_i\big) \Big)(L_i x -r_i) + h(x) -\<x,z\> \bigg\}
\end{align}
together with its conjugate dual problem
\begin{align}
	\label{ic_p_dual_convex}
	\sup_{\substack{(\f p, \f q) \in \f \X \oplus \f \Y,\\K_i^*p_i = M_i^*q_i,\ i=1,\ldots,m}}\!\!\bigg\{- (f^* \Box h^*)\left(z-\sum_{i=1}^m L_i^*K_i^*p_i\right) - \sum_{i=1}^m\Big[g_i^*(p_i)+ l_i^*(q_i) +\<p_i,K_ir_i\>\Big]\bigg\}.
\end{align}
\end{problem}
For every $x \in \h$ and $(\f p, \f q) \in \f \X \oplus \f \Y$ with $K_i^*p_i = M_i^*q_i$, $i=1,\ldots,m$, by the Young-Fenchel inequality, it holds
$$f(x) + h(x) +  (f^* \Box h^*)\left(z-\sum_{i=1}^m L_i^*K_i^*p_i\right) \geq \left \langle z-\sum_{i=1}^m L_i^*K_i^*p_i, x \right \rangle $$
and, for any $i=1,\ldots,m$ and $y_i \in \g$,
$$g_i(K_i(L_ix - r_i -y_i)) + g_i^*(p_i) \geq \langle p_i, K_i(L_ix - r_i -y_i)  \rangle = \langle K_i^*p_i, L_ix - r_i -y_i \rangle$$
and
$$l_i(M_iy_i) + l_i^*(q_i) \geq \langle q_i, M_iy_i \rangle = \langle M_i^*q_i, y_i\rangle.$$
This yields
\begin{align}\label{ic_wd}
&\inf_{x\in \h}\bigg\{ f(x) + \sum_{i=1}^m \Big( \big(g_i \circ K_i\big) \Box \big(l_i \circ M_i\big) \Big)(L_i x -r_i) + h(x) -\<x,z\> \bigg\} \nonumber\\
 = &\inf_{(x,\f y)\in \h \oplus \f \g}\bigg\{ f(x) + \sum_{i=1}^m \Big( g_i (K_i(L_ix-r_i-y_i)) + l_i(M_iy_i) \Big) + h(x) -\<x,z\> \bigg\} \\
\geq  & \!\!\!\sup_{\substack{(\f p, \f q) \in \f \X \oplus \f \Y,\\K_i^*p_i = M_i^*q_i,\ i=1,\ldots,m}}\!\!\bigg\{- (f^* \Box h^*)\left(z-\sum_{i=1}^m L_i^*K_i^*p_i\right) - \sum_{i=1}^m\Big[g_i^*(p_i)+ l_i^*(q_i) +\<p_i,K_ir_i\>\Big]\bigg\},\nonumber
\end{align}
which means that for the primal-dual pair of optimization problems \eqref{ic_p_primal_convex}-\eqref{ic_p_dual_convex} weak duality is always given.

Considering $(\bx, \f \bp, \f \bq, \f \by) \in \h \oplus \f \X \oplus \f \Y \oplus \f \g$  a solution of the primal-dual system of monotone inclusions
\begin{align}
\label{ic_condition_convex}
\begin{aligned}
& z - \sum_{i=1}^mL_i^*K_i^*\bp_i \in \partial f(\bx) +\nabla h(\bx) \ \mbox{and} \\
& K_i(L_i\bx -\by_i-r_i) \in \partial g_i^*(\bp_i),\ M_i\by_i \in \partial l_i^{*}(\bq_i),\  K_i^*\bp_i = M_i^*\bq_i,\ i=1,\ldots,m.
\end{aligned}
\end{align}
it follows that $\bx$ is an optimal solution to \eqref{ic_p_primal_convex} and that $(\f \bp, \f \bq)$ is an optimal solution to \eqref{ic_p_dual_convex}. Indeed, as $h$ is convex and everywhere differentiable, it holds
$$z - \sum_{i=1}^mL_i^*K_i^*\bp_i \in \partial f(\bx) +\nabla h(\bx) \subseteq \partial (f+h)(\bx),$$
thus,
$$f(\bx) + h(\bx) + (f^* \Box h^*)\left(z-\sum_{i=1}^m L_i^*K_i^*\bp_i\right) = \left \langle z-\sum_{i=1}^m L_i^*K_i^*\bp_i, \bx \right \rangle.$$
On the other hand, since $g_i \in \Gamma(\X_i)$ and $l_i \in \Gamma(\Y_i)$, we have for any $i=1,\ldots,m$
$$g_i(K_i(L_i\bx -\by_i-r_i)) + g_i^*(\bp_i) =  \langle K_i^*\bp_i, L_i\bx - r_i -\by_i \rangle$$
and
$$l_i(M_i\by_i) + l_i^*(\bq_i) = \langle M_i^*\bq_i, \by_i\rangle.$$
By summing up these equations and using \eqref{ic_condition_convex}, it yields
\begin{align*}
&f(\bx) + \sum_{i=1}^m \Big( \big(g_i \circ K_i\big) \Box \big(l_i \circ M_i\big) \Big)(L_i \bx -r_i) + h(\bx) -\<\bx,z\>  \\
  \leq\ &f(\bx) + \sum_{i=1}^m \Big( g_i (K_i(L_i\bx-r_i-\by_i)) + l_i(M_i\by_i) \Big) + h(\bx) -\<\bx,z\>   \\
 =\ &- (f^* \Box h^*)\left(z-\sum_{i=1}^m L_i^*K_i^*\bp_i\right) - \sum_{i=1}^m\Big[g_i^*(\bp_i)+ l_i^*(\bq_i) +\<\bp_i,K_ir_i\>\Big],
\end{align*}
which, together with \eqref{ic_wd}, leads to the desired conclusion.

In the following, by extending the result in \cite[Proposition 4.2]{BecCom13} to our setting, we provide sufficient conditions which guarantee the validity of \eqref{ic_th1_inclusion} when applied to convex minimization problems. To this end we mention that the \textit{strong quasi-relative interior} of a nonempty convex set $\Omega \subseteq \h$ is defined as
$$\sqri \Omega=\left \{x \in \Omega: \bigcup_{\lambda \geq 0} \ \lambda(\Omega-x) \ \mbox{is a closed linear subspace} \right\}.$$

\begin{proposition}
	\label{ic_th_ran_convex}
Suppose that the primal problem \eqref{ic_p_primal_convex} has an optimal solution, that
	\begin{align}
		\label{ic_sri_convex}
		0 \in \sqri \left( \dom (g_i \circ K_i)^* - \dom (l_i \circ M_i)^*\right),\ i=1,\ldots,m
	\end{align}
and
\begin{align}\label{ic_E_convex}
0 \in \sqri \f E,
\end{align}
where
\begin{align*}
\f E\!:=\!\!\bigg\{\!\bigtimes\limits_{i=1}^m\!\Big\{K_i(L_i(\dom f)\!-\!r_i\!-\!y_i)\!-\!\dom g_i \Big\} \!\times \!\bigtimes\limits_{i=1}^m\!\Big\{M_iy_i\!-\!\dom l_i\Big\} : y_i \in \g_i,\, i=1,...,m \bigg\}.	
\end{align*}

Then 
\begin{align*}
	z \in \ran \bigg( \partial f +  \sum_{i=1}^m L_i^*\left((K_i^*\circ \partial g_i \circ K_i )\Box (M_i^*\circ \partial l_i \circ M_i)\right)(L_i \cdot - r_i) +\nabla h \bigg).
\end{align*}
\end{proposition}

\begin{proof}
Let $\bx \in \h$ be an optimal solution to  \eqref{ic_p_primal_convex}. Since \eqref{ic_E_convex} holds, we have that $(g_i\circ K_i)$, $(l_i\circ M_i) \in \Gamma(\g_i)$, $i=1,\ldots,m$. Further, because of \eqref{ic_sri_convex}, 
\cite[Proposition 15.7]{BauCom11} guarantees for any $i=1,\ldots,m$ the existence of $\by_i \in \g_i$ such that 
$$ \big((g_i \circ K_i) \Box (l_i \circ M_i)\big)(\bx) = (g_i\circ K_i)(\bx-\by_i) + (l_i \circ M_i)(\by_i).$$
Hence, $(\bx, \f \by) = (\bx,\by_1,\ldots,\by_m)$ is an optimal solution to the convex optimization problem
\begin{align}
	\label{ic_p_primal2_convex}
	\inf_{(x,\f y)\in \h\oplus \f \g}\bigg\{ f(x) + h(x) -\<x,z\> + \sum_{i=1}^m \Big[ g_i(K_i(L_ix-r_i-y_i)) + l_i(M_iy_i) \Big] \bigg\}
\end{align}
By denoting
\begin{align}
	\label{ic_notations_sri_convex}
	\begin{aligned}
	\f f: \h \oplus \f \g &\rightarrow \oR,\ \f f(x,\f y) = f(x) + h(x) - \<x,z\> \\
	\f g:\f \X \oplus \f \Y &\rightarrow \oR,\ \f g(\f x, \f y)= \sum_{i=1}^m \Big[ g_i(x_i-K_ir_i) + l_i(y_i) \Big] \\
	\f L: \h \oplus \f \g &\rightarrow \f \X \oplus \f \Y,\ (x, \f y) \mapsto \bigtimes_{i=1}^m\Big\{ K_i(L_ix-y_i) \Big\} \times \bigtimes_{i=1}^m \Big\{ M_i y_i \Big\}, \\
   \end{aligned}
\end{align}

problem \eqref{ic_p_primal2_convex} can be equivalently written as
\begin{align}
	\label{ic_p_primal3_convex}
	\inf_{(x,\f y)\in\h\oplus \f \g}\left\{ \f f(x,\f y) + \f g(\f L(x,\f y)) \right\}.
\end{align}
Thus,
$$0 \in \partial(\f f + \f g\circ \f L)(\bx, \f \by).$$
Since $\f E = \f L(\dom \f f) - \dom \f g$ and \eqref{ic_E_convex} is fulfilled, it holds (see, for instance, \cite{Bot10, BGW09, BauCom11})
$$0 \in \partial \big( \f f + \f g \circ \f L \big)(\bx, \f \by) = \partial \f f(\bx, \f \by) + \big(\f L^* \circ \partial \f g \circ \f L\big)(\bx, \f \by),$$ 
where
$$\f L^* : \f \X \oplus \f \Y \rightarrow \h \oplus \f \g,\ (\f p, \f q) \mapsto  \Big( \sum_{i=1}^m L_i^*K_i^*p_i, -K_1^*p_1 + M_1^*q_1,\ldots, -K_m^*p_m + M_m^*q_m \Big).$$
We obtain
{\allowdisplaybreaks
\begin{align*}
	& 0 \in \partial \f f(\bx,\f\by) + \big( \f L^* \circ \partial \f g \circ \f L \big)(\bx,\f \by) \\
	 \Leftrightarrow & \ \left\{
	\begin{array}{l}
	0 \in \partial f(\bx) + \nabla h(\bx) -z + \sum_{i=1}^m L_i^*\big(K_i^*\circ \partial g_i \circ K_i \big)(L_i \bx -r_i - \by_i)\\
	0 \in -\big(K_i^*\circ \partial g_i \circ K_i\big)(L_i \bx -r_i - \by_i) + \big(M_i^* \circ \partial l_i \circ M_i \big) \by_i,\ i=1,\ldots,m
	\end{array}\right. \\
	\Leftrightarrow & \ \exists \f v \in \f \g: \left\{
	\begin{array}{l}
	0 \in \partial f(\bx) + \nabla h(\bx) -z + \sum_{i=1}^m L_i^* v_i \\
	v_i \in \big( K_i^* \circ \partial g_i \circ K_i \big)(L_i \bx -r_i - \by_i),\ i=1,\ldots,m \\
	v_i \in \big( M_i^* \circ \partial l_i \circ M_i \big)\by_i,\ i=1,\ldots,m
	\end{array}\right. \\
	\Leftrightarrow & \ \exists \f v \in \f \g: \left\{
	\begin{array}{l}
	0 \in \partial f(\bx) + \nabla h(\bx) -z + \sum_{i=1}^m L_i^* v_i \\
	L_i \bx -r_i - \by_i \in \big(K_i^*\circ \partial g_i \circ K_i\big)^{-1} v_i,\ i=1,\ldots,m \\
	\by_i \in \big(M_i^* \circ \partial l_i \circ M_i\big)^{-1} v_i,\ i=1,\ldots,m
	\end{array}\right.	\\
	\Leftrightarrow & \ \exists \f v \in \f \g: \left\{
	\begin{array}{l}
	0 \in \partial f(\bx) + \nabla h(\bx) -z + \sum_{i=1}^m L_i^* v_i \\
	v_i \in \Big( \big(K_i^* \circ \partial g_i \circ K_i \big) \Box \big(M_i^* \circ \partial l_i \circ M_i\big) \Big) (L_i \bx - r_i),\ i=1,\ldots,m
	\end{array}\right.	\\
	\Leftrightarrow  & \ z \in \partial f(\bx) + \sum_{i=1}^m L_i^* \Big( \big(K_i^* \circ \partial g_i \circ K_i \big) \Box \big(M_i^* \circ \partial l_i \circ M_i\big) \Big) (L_i \bx - r_i) + \nabla h(\bx),
\end{align*}}
which completes the proof.
\end{proof}

\begin{remark}\label{ic_finalremark_convex}
If one of the following two conditions 
\begin{enumerate}
\setlength{\itemsep}{-2pt}
		\item[$\bullet$] $f$ is real-valued and the operators $L_i$, $K_i$ and $M_i$ are surjective for any $i =1,\ldots,m$;
		\item[$\bullet$] the functions $g_i$ and $l_i$ are real-valued for any $i=1,\ldots,m$;		
\end{enumerate}	
is fulfilled, then $\f E = \f \X \oplus \f \Y$ and \eqref{ic_E_convex} is obviously true. \\
On the other hand, if $\h$, $\g_i, \X_i$ and $\Y_i$, $i=1,\ldots,m$ are finite dimensional and
\begin{align*}
			\mbox{for any} \ i =1,\ldots,m \ \mbox{exists} \ y_i \in \g_i: \left\{
             \begin{array}{l} 
            K_iy_i \in K_i(L_i(\ri \dom f) - r_i) - \ri \dom g_i,\\ 
            M_iy_i \in \ri \dom l_i
            \end{array}\right.,
\end{align*}
then \eqref{ic_E_convex} is also true. This follows by using that in finite dimensional spaces the strong quasi-relative interior of a convex set is nothing else than its relative interior and by  taking into account the properties of the latter.
\end{remark}

\subsection{An algorithm of forward-backward type}\label{ic_subsecFB_convex}
When applied to \eqref{ic_condition_convex}, the iterative scheme introduced in \eqref{ic_A1} and the corresponding convergence statements read as follows.
\begin{algorithm}\label{ic_alg1_convex} \text{ }\newline
Let $x_0 \in \h$, and for any $i=1,\ldots,m$, let $p_{i,0}\in\X_i$, $q_{i,0} \in \Y_i$ and $y_{i,0},\ z_{i,0},\ v_{i,0} \in \g_i$. For any $i=1,\ldots,m$, let $\tau,\,\theta_{1,i},\, \theta_{2,i},\, \gamma_{1,i},\, \gamma_{2,i}$ and $\sigma_i$ be strictly positive real numbers such that
\begin{align}
	\label{ic_condition_step_length_convex}
	2\mu^{-1}\left(1-\overline{\alpha} \right) \min_{i=1,\ldots,m}\bigg\{\frac{1}{\tau},\frac{1}{\theta_{1,i}},\frac{1}{\theta_{2,i}}, \frac{1}{\gamma_{1,i}},\frac{1}{\gamma_{2,i}},\frac{1}{\sigma_i} \bigg\}>1,
\end{align}
for
\begin{align*}
\overline{\alpha} = \max\left\{\sqrt{\tau \sum_{i=1}^m \sigma_i \|L_i\|^2}, \max_{j=1,\ldots,m}\left\{\sqrt{\theta_{1,j}\gamma_{1,j}\|K_j\|^2}, \sqrt{\theta_{2,j}\gamma_{2,j}\|M_j\|^2}\right\} \right\}.
\end{align*}
Furthermore, let $\varepsilon \in (0,1)$, $(\lambda_n)_{n\geq 0}$ a sequence in $\left[\varepsilon,1\right]$ and set
{\allowdisplaybreaks	
\begin{align}\label{A1_convex}
	  \left(\forall n\geq 0\right) \begin{array}{l} \left\lfloor \begin{array}{l}
		\wt x_{n} \approx \Prox_{\tau f}\left( x_n - \tau\left(Cx_n + \sum_{i=1}^m L_i^* v_{i,n} - z\right) \right) \\
		\text{For }i=1,\ldots,m  \\
				\left\lfloor \begin{array}{l}
					\wt p_{i,n} \approx \Prox_{\theta_{1,i} g_i^{*}}\left(p_{i,n} +\theta_{1,i} K_i z_{i,n}  \right) \\
					\wt q_{i,n} \approx \Prox_{\theta_{2,i} l_i^{*}}\left(q_{i,n} +\theta_{2,i} M_i y_{i,n}  \right) \\
					u_{1,i,n} \approx z_{i,n} +\gamma_{1,i}\left(K_i^*\left(p_{i,n}-2\wt p_{i,n}\right)+v_{i,n}+\sigma_i\left(L_i(2\wt x_{n}-x_n)-r_i\right)\right) \\
					u_{2,i,n} \approx y_{i,n} +\gamma_{2,i}\left(M_i^*\left(q_{i,n}-2\wt q_{i,n}\right)+v_{i,n}+\sigma_i\left(L_i(2\wt x_{n}-x_n)-r_i\right)\right) \\
					\wt z_{i,n} \approx \frac{1+\sigma_i\gamma_{2,i}}{1+\sigma_i(\gamma_{1,i}+\gamma_{2,i})}\left(u_{1,i,n}-\frac{\sigma_i\gamma_{1,i}}{1+\sigma_i\gamma_{2,i}} u_{2,i,n}\right) \\
					\wt y_{i,n} \approx \frac{1}{1+\sigma_i\gamma_{2,i}}\left(u_{2,i,n} -\sigma_i\gamma_{2,i}\wt z_{i,n}\right) \\
					\wt v_{i,n} \approx v_{i,n} + \sigma_i\left(L_i(2\wt x_{n}-x_n)-r_i- \wt z_{i,n} - \wt y_{i,n}\right)
				\end{array} \right.\\
		x_{n+1} = x_n + \lambda_n (\wt x_{n}-x_n ) \\
		\text{For }i=1,\ldots,m  \\
				\left\lfloor \begin{array}{l}
					p_{i,n+1} = p_{i,n} + \lambda_n(\wt p_{i,n}-p_{i,n})\\
					q_{i,n+1} = q_{i,n} + \lambda_n(\wt q_{i,n}-q_{i,n})\\
					z_{i,n+1} = z_{i,n} + \lambda_n(\wt z_{i,n}-z_{i,n})\\
					y_{i,n+1} = y_{i,n} + \lambda_n(\wt y_{i,n}-y_{i,n})\\
					v_{i,n+1} = v_{i,n} + \lambda_n(\wt v_{i,n}-v_{i,n}).
				\end{array} \right. \\ \vspace{-4mm}
		\end{array}
		\right.
		\end{array}
	\end{align}}
\end{algorithm}

\begin{theorem}\label{ic_th1_convex}
For Problem \ref{ic_problem_convex}, suppose that
\begin{align}\label{ic_th1_inclusion_convex}
	z \in \ran \bigg( \partial f +  \sum_{i=1}^m L_i^*\left((K_i^*\circ \partial g_i \circ K_i )\Box (M_i^*\circ \partial l_i \circ M_i)\right)(L_i \cdot - r_i) +\nabla h \bigg)
\end{align}
and consider the sequences generated by Algorithm \ref{ic_alg1_convex}. Then there exists an optimal solution $\bx$ to \eqref{ic_p_primal_convex} and optimal solution $(\bp,\bq)$ to \eqref{ic_p_dual_convex} such that
\begin{enumerate}[label={(\roman*)}]
	\setlength{\itemsep}{-2pt}
		\item  \label{ic_th1.01_convex} $x_n \rightharpoonup \bx$, $p_{i,n} \rightharpoonup \bp_{i}$ and $q_{i,n} \rightharpoonup \bq_{i}$ for any $i=1,\ldots,m$ as $n \rightarrow +\infty$.
    \item  \label{ic_th1.02_convex} if $h$ is uniformly convex at $\bx$, then $x_n \rightarrow \bx$ as $n \rightarrow +\infty$.
	\end{enumerate}
\end{theorem}

\begin{proof}
The results is a direct consequence of Theorem \ref{ic_th1} when taking
	\begin{align}
			\label{ic_operators_convex}
			A=\partial f,\ C = \nabla h, \text{ and } \ B_i=\partial g_i,\ D_i =\partial l_i,\ i=1,\ldots,m.
	\end{align}
We also notice that, according to Theorem 20.40 in \cite{BauCom11}, the operators in \eqref{ic_operators_convex} are maximally monotone, while, by \cite[Corollary 16.24]{BauCom11}, we have $A^{-1}=\partial f^*$, $C^{-1}=\partial h^*$, $B_i^{-1}=\partial g_i^*$ and $D_i^{-1}=\partial l_i^*$ for $i=1,\ldots,m$. Furthermore, by \cite[Corollary 18.16]{BauCom11}, $C=\nabla h$ is $\mu^{-1}$-cocoercive, while, if $h$ is uniformly convex at $\bx \in \h$, then $C=\nabla h$ is uniformly monotone at $\bx$ (cf. \cite[Section 3.4]{Zal02}).
\end{proof}

\begin{remark}
	\label{ic_remark_C0_convex}
If $h \in \Gamma(\h)$ such that $\nabla h(x)=0$ for all $x\in\h$, then condition \eqref{ic_condition_step_length_convex} simplifies to
	\begin{align*}
		\max\bigg\{\tau \sum_{i=1}^m \sigma_i \|L_i\|^2, \max_{j\in\I}\left\{\theta_{1,j}\gamma_{1,j}\|K_j\|^2, \theta_{2,j}\gamma_{2,j}\|M_j\|^2\right\} \bigg\}<1.
	\end{align*}
In this situation Algorithm \ref{ic_alg1_convex} converges under the relaxed assumption that $(\lambda_n)_{n\geq 0} \subseteq \left[\varepsilon,2-\varepsilon\right]$ for $\varepsilon \in (0,1)$ (see also Remark \ref{ic_remark_C0}).
\end{remark}

\subsection{An algorithm of forward-backward-forward type}\label{ic_subsecFBF_convex}
On the other hand, when applied to \eqref{ic_condition_convex}, the iterative scheme introduced in \eqref{ic_A1_2} and the corresponding convergence statements read as follows.
\begin{algorithm}\label{ic_alg2_convex} \text{ }\newline
Let $x_0 \in \h$, and for any $i=1,\ldots,m$, let $p_{i,0}\in\X_i$, $q_{i,0} \in \Y_i$, and $z_{i,0},\ y_{i,0},\ v_{i,0} \in \g_i$. Set
\begin{align}
	\label{ic_condition_beta_convex}
	\beta = \mu + \sqrt{\max \bigg\{ \sum_{i=1}^m\|L_i\|^2, \max_{j=1,...,m}\big\{ \|K_j\|^2, \|M_j\|^2 \big\}  \bigg\}},
\end{align}
 let $\varepsilon \in \left (0,\frac{1}{\beta+1}\right)$, $(\gamma_n)_{n\geq 0}$ a sequence in $\left[\varepsilon,\frac{1-\varepsilon}{\beta}\right]$ and set
	\begin{align}\label{ic_A1_2_convex}
	  \left(\forall n\geq 0\right) \begin{array}{l} \left\lfloor \begin{array}{l}
		\wt x_n \approx \Prox_{\gamma_n f}\left( x_n - \gamma_n\left(Cx_n + \sum_{i=1}^m L_i^* v_{i,n} - z\right) \right) \\
		\text{For }i=1,\ldots,m  \\
				\left\lfloor \begin{array}{l}
					\wt p_{i,n} \approx \Prox_{\gamma_n g_i^*}\left(p_{i,n} +\gamma_n K_i z_{i,n}  \right) \\
					\wt q_{i,n} \approx \Prox_{\gamma_n l_i^*}\left(q_{i,n} +\gamma_n M_i y_{i,n}  \right) \\
					u_{1,i,n} \approx z_{i,n} -\gamma_n\left(K_i^*p_{i,n}-v_{i,n}-\gamma_n\left(L_ix_n-r_i\right)\right) \\
					u_{2,i,n} \approx y_{i,n} -\gamma_n\left(M_i^*q_{i,n}-v_{i,n}-\gamma_n\left(L_ix_n-r_i\right)\right) \\
					\wt z_{i,n} \approx \frac{1+\gamma_n^2}{1+2\gamma_n^2}\left(u_{1,i,n}-\frac{\gamma_n^2}{1+\gamma_n^2} u_{2,i,n}\right) \\
					\wt y_{i,n} \approx \frac{1}{1+\gamma_n^2}\left(u_{2,i,n} -\gamma_n^2 \wt z_{i,n}\right) \\
					\wt v_{i,n} \approx v_{i,n} + \gamma_n\left(L_ix_n-r_i - \wt z_{i,n}- \wt y_{i,n}\right) 
				\end{array} \right.\\
		x_{n+1} \approx \wt x_n + \gamma_n(Cx_n-C \wt x_n + \sum_{i=1}^m L_i^*(v_{i,n}-\wt v_{i,n})) \\
		\text{For }i=1,\ldots,m  \\
				\left\lfloor \begin{array}{l}
					p_{i,n+1} \approx \wt p_{i,n} - \gamma_n(K_i(z_{i,n}-\wt z_{i,n}))\\
					q_{i,n+1} \approx \wt q_{i,n} - \gamma_n(M_i(y_{i,n}-\wt y_{i,n}))\\
					z_{i,n+1} \approx \wt z_{i,n} + \gamma_n(K_i^*(p_{i,n}-\wt p_{i,n}))\\
					y_{i,n+1} \approx \wt y_{i,n} + \gamma_n(M_i^*(q_{i,n}-\wt q_{i,n}))\\
					v_{i,n+1} \approx \wt v_{i,n} - \gamma_n(L_i(x_n- \wt x_n)).
				\end{array} \right. \\ \vspace{-4mm}
		\end{array}
		\right.
		\end{array}
	\end{align}
\end{algorithm}

\begin{theorem}\label{ic_th2_convex}
For Problem \ref{ic_problem_convex}, suppose that
\begin{align}\label{ic_th2_inclusion_convex}
	z \in \ran \bigg( \partial f +  \sum_{i=1}^m L_i^*\Big(\big(K_i^*\circ \partial g_i \circ K_i \big)\Box \big(M_i^*\circ \partial l_i\circ M_i\big)\Big)(L_i \cdot -r_i) +\nabla h \bigg),
\end{align}
and consider the sequences generated by Algorithm \ref{ic_alg2_convex}. Then there exists an optimal solution $\bx$ to \eqref{ic_p_primal_convex} and optimal solution $(\bp,\bq)$ to \eqref{ic_p_dual_convex} such that
\begin{enumerate}[label={(\roman*)}]
	\setlength{\itemsep}{-2pt}
		\item  \label{ic_th2.01_convex} $\sum_{n\geq 0}\|x_n- \wt x_n\|^2 < +\infty$ and for any $i=1,...,n$
		\begin{align*}
			\sum_{n\geq 0} \|p_{i,n} - \wt p_{i,n}\|^2 < +\infty \text{ and } \sum_{n\geq 0} \|q_{i,n} - \wt q_{i,n}\|^2 < +\infty.
		\end{align*}
		\item  \label{ic_th2.02_convex} $x_n \rightharpoonup \bx$, $\wt x_n \rightharpoonup \bx$ and for any $i=1,...,n$
		\begin{align*}
			 \left\{\begin{array}{l} p_{i,n}\rightharpoonup \bp_{i,n} \\ \wt p_{i,n}\rightharpoonup \bp_{i,n} \end{array}\right. \text{ and }
		   \left\{\begin{array}{l} q_{i,n}\rightharpoonup \bq_{i,n} \\ \wt q_{i,n}\rightharpoonup \bq_{i,n} \end{array}\right.\!\!.
		\end{align*}
	\end{enumerate}
\end{theorem}
\begin{proof}
The conclusions follow by using the statements in the proof of Theorem \ref{ic_th1_convex} and by applying Theorem \ref{ic_th2}. 
\end{proof}

\section{Numerical experiments}\label{ic_secExperiment}

Within this section we solve image denoising problems where first- and second-order total variation functionals are linked via infimal convolutions. This approach  was initially proposed in \cite{ChaLio97} and further 
investigated in \cite{SetSteTeu11}. 

Let $b\in\R^n$ be the observed and vectorized noisy image of size $M \times N$ (with $n=MN$ for greyscale and $n=3MN$ for colored images). 
For $k \in \N$ and $\omega=(\omega_1,\ldots,\omega_k)\in \R_{++}^k$ we consider on $\R^{k \times n}$ the following norm defined for $y=(y_1,\ldots,y_k)^T \in \R^{k \times n}$ 
as
\begin{align*}
	\|y\|_{1,\omega} = \Big\|\big(\omega_1 y_1^2 + \ldots + \omega_k y_k^2 \big)^{\frac{1}{2}}\Big\|_1,
\end{align*}
where addition, multiplication and square root of vectors are understood to be componentwise. Further, we consider the forward difference matrix 
{\scriptsize
\begin{align*}
		D_k:=\left[\begin{array}{rrrrrr}
		-1 & 1 & 0 & 0 & \cdots & 0 \\
		0 & -1 & 1 & 0 & \cdots & 0 \\
		\vdots & & \ddots & \ddots & & \vdots \\
		0 & \cdots & 0 & -1 & 1 & 0\\
		0 & \cdots & 0 & 0 & -1 & \ 1 \\
		0 & \cdots & 0 & 0 & 0 & 0
		\end{array}\right] \in \R^{k\times k},
\end{align*}}

\noindent which models the discrete first-order derivative. Note that $-D_k^TD_k$ is then an approximation of the second-order derivative. We denote by $A \otimes B$ the Kronecker product of the matrices $A$ and $B$ and 
define
\begin{align}
	\label{ic_example_Kronecker_0}
	D_x = \id_N \otimes D_M, \ D_y = D_N \otimes \id_M \ \mbox{and} \
	\D_1=\left[\begin{array}{r} D_x \\ D_y \end{array}\right], 
\end{align}
where $D_x$ and $D_y$ represent the vertical and horizontal difference operators, respectively. Further, we define the discrete second-order derivatives matrices
\begin{align}
	\label{ic_example_Kronecker}
	 D_{xx} = \id_N \otimes (-D_M^TD_M), \ D_{yy} = (-D_N^TD_N) \otimes \id_M, \ 
	\D_2=\left[\begin{array}{r} D_{xx} \\ D_{yy} \end{array}\right]
\end{align}
and
\begin{align*}
	L_1=\left[\begin{array}{cc} -D_x^T & 0 \\ 0 & -D_y^T \end{array}\right]
\end{align*}
and notice that $\D_2=L_1 \D_1$. For other discrete second-order derivates involving also mixed partial derivates (in horizontal-vertical direction and vice versa) we refer to \cite{SetSteTeu11}.

The two different convex optimization problems we considered for our numerical experiments were taken from \cite[Example 2.2 and Example 3.1]{SetSteTeu11} and readed
\begin{align}
	\label{ic_problem_IC}
	(\ell_2^2\text{-IC/P}) \quad \inf_{x\in \R^n}\left\{\frac{1}{2}\|x-b\|^2 + \Big((\alpha_1\|\cdot\|_{1,\omega_1} \circ \D_1) \Box (\alpha_2\|\cdot\|_{1,\omega_2} \circ \D_2)\Big)(x)\right\},
\end{align}
and
\begin{align}
	\label{ic_problem_MIC}
	(\ell_2^2\text{-MIC/P}) \quad \inf_{x\in \R^n}\left\{\frac{1}{2}\|x-b\|^2 + \Big((\alpha_1\|\cdot\|_{1,\omega_1}) \Box (\alpha_2\|\cdot\|_{1,\omega_2} \circ L_1)\Big)(\D_1x)\right\},
\end{align}
respectively, where $\alpha_1, \alpha_2 \in \R_{++}$ are the regularization parameters and the regularizers correspond to anistropic total variation functionals. One can notice that in both settings a condition of type 
\eqref{ic_sri_convex} is fulfilled, thus the infimal convolutions are proper, convex and lower semicontinuous functions. Due to the fact that the objective functions of the two optimization problems are proper, strongly convex
and lower semicontinuous, they have unique optimal solutions. Finally, in the light of Remark \ref{ic_finalremark_convex}, a condition of type \eqref{ic_E_convex} holds, thus, according to Proposition \ref{ic_th_ran_convex},
the hypotheses of the theorems \ref{ic_th1_convex} and \ref{ic_th2_convex} are for both optimization problems $(\ell_2^2\text{-IC/P})$ and $(\ell_2^2\text{-MIC/P})$ fulfilled.

In order to compare the performances of our two primal-dual iterative schemes with algorithms relying on (augmented) Lagrangian and smoothing techniques, using the definition of the infimal convolution, we formulated
\eqref{ic_problem_IC} and \eqref{ic_problem_MIC} as optimization problems with constraints of the form  
\begin{align}
	\label{ic_problem_IC_2}
	\begin{aligned}
	(\ell_2^2\text{-IC/P}) \quad \inf_{x_1,x_2,z_1,z_2}&\left\{\frac{1}{2}\|x_1+x_2-b\|^2 + \alpha_1\|z_1\|_{1,\omega_1} +\alpha_2\|z_2\|_{1,\omega_2}\right\}, \\
	\text{subject to } &\left(\begin{array}{cc}\D_1 & 0 \\ 0 & \D_2 \end{array}\right) \left(\begin{array}{c} x_1 \\ x_2 \end{array}\right) = \left(\begin{array}{c} z_1 \\ z_2 \end{array}\right)
	\end{aligned}
\end{align}
and
\begin{align}
	\label{ic_problem_MIC_2}
	\begin{aligned}
	(\ell_2^2\text{-MIC/P}) \quad \inf_{x,y_1,y_2,z}&\left\{\frac{1}{2}\|x-b\|^2 + \alpha_1\|y_1\|_{1,\omega_1} +\alpha_2\|z\|_{1,\omega_2}\right\}, \\
	\text{subject to } &\left(\begin{array}{cc}\D_1 & -\id \\ 0 & L_1 \end{array}\right) \left(\begin{array}{c} x \\ y_2 \end{array}\right) = \left(\begin{array}{c} y_1 \\ z \end{array}\right)
	\end{aligned}
\end{align}
respectively.

\begin{figure}[tb]	
	\centering
	\captionsetup[subfigure]{position=top}
	\subfloat[Original image]{\includegraphics*[viewport=144 249 486 574, width=0.32\textwidth]{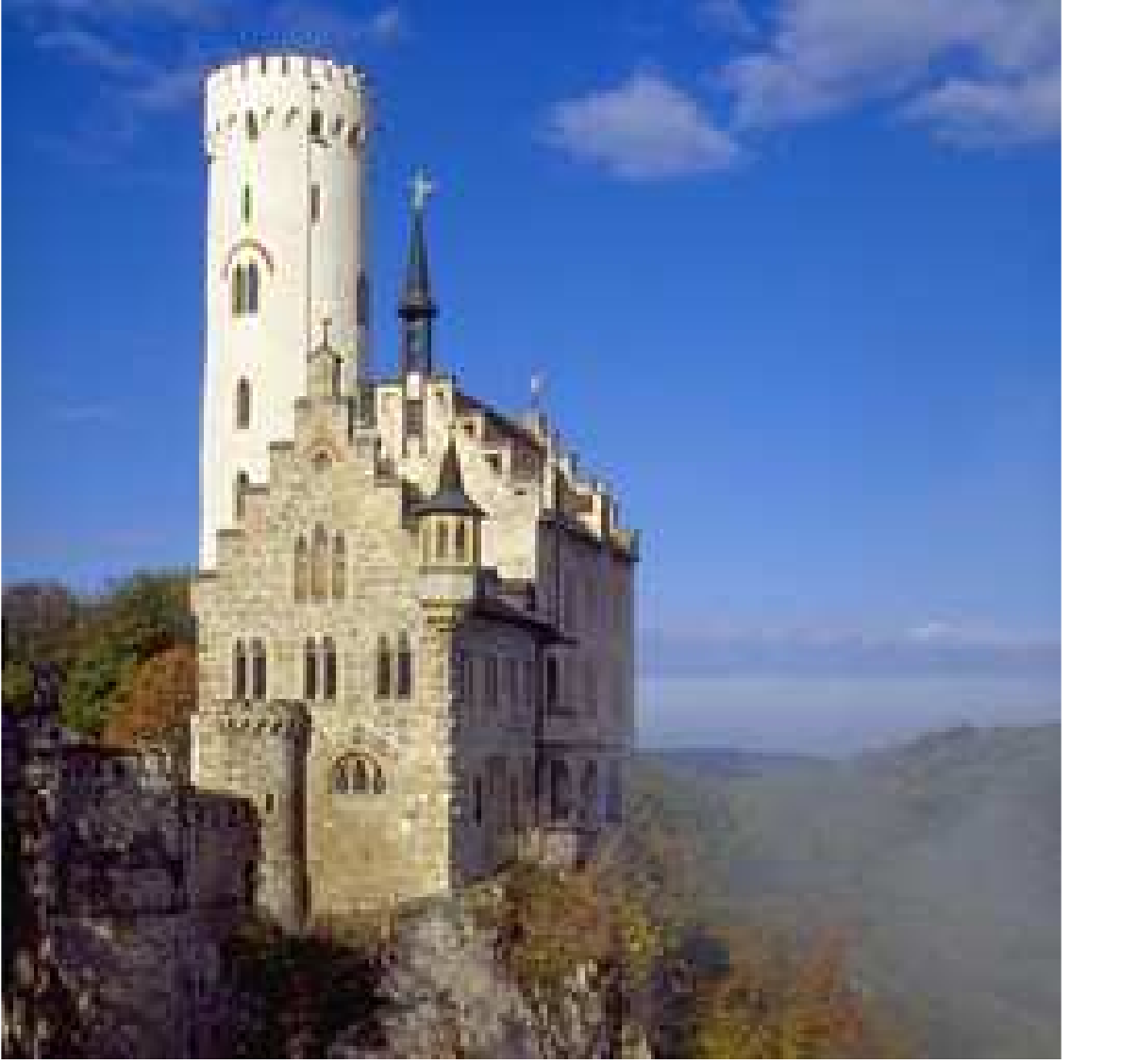}} \hspace{0.2mm}
	\subfloat[Noisy image]{\includegraphics*[viewport=144 249 486 574, width=0.32\textwidth]{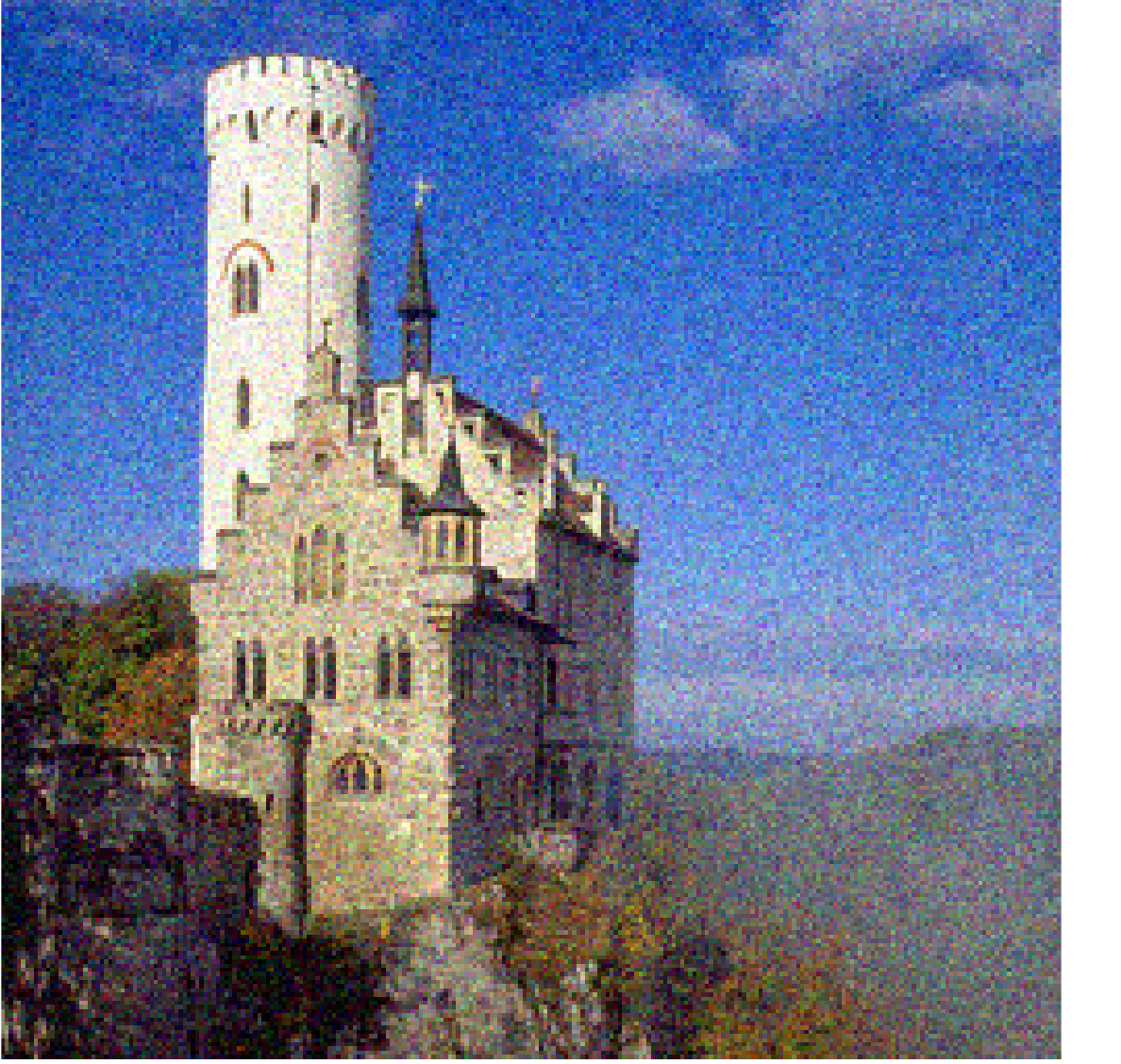}}  \hspace{0.2mm}
	\subfloat[Reconstructed image]{\includegraphics*[viewport=144 249 486 574, width=0.32\textwidth]{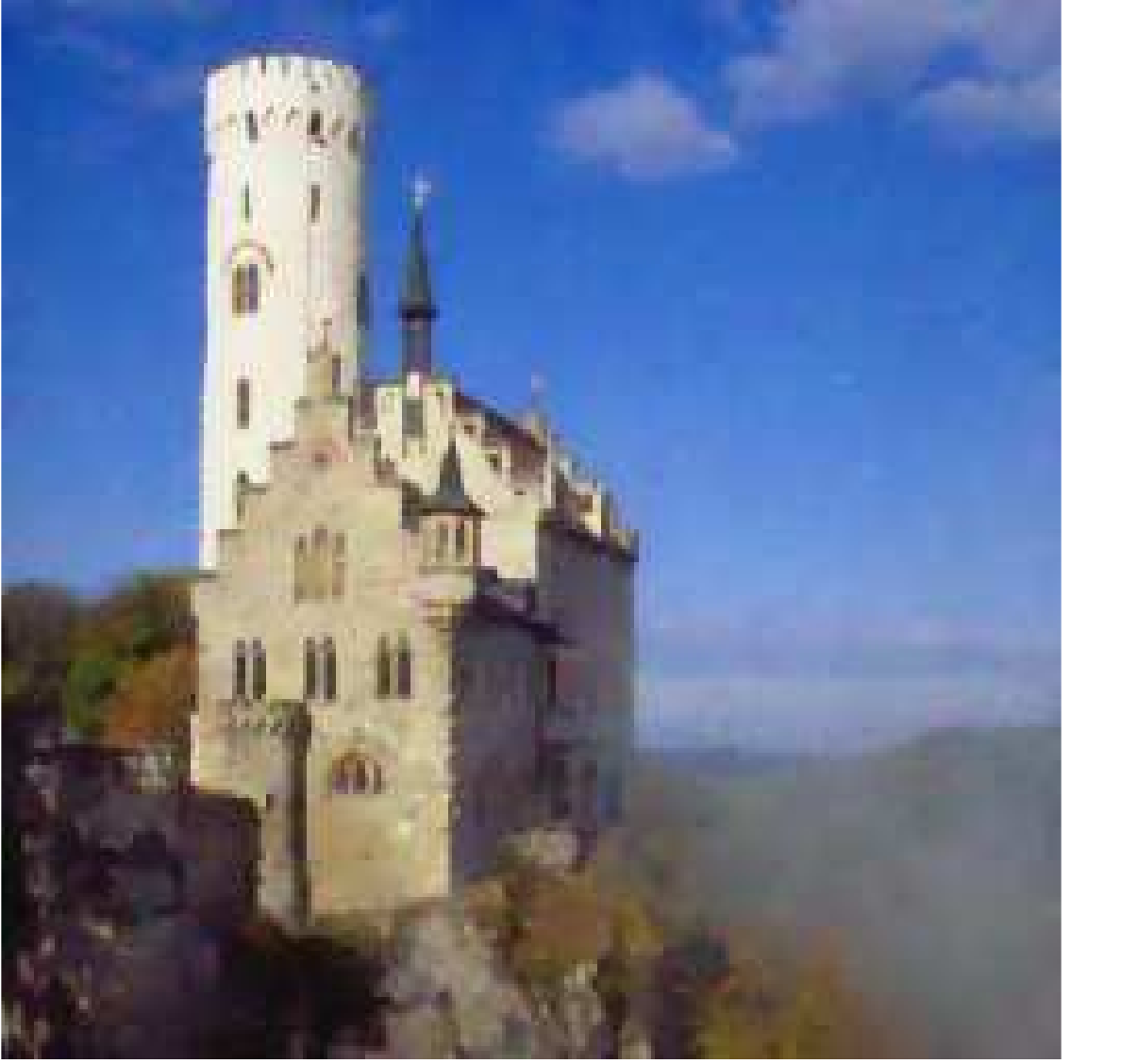}}	
	\caption{\small Figure (a) shows the clean $256\times 256$ lichtenstein test image, (b) shows the image obtained after adding white Gaussian noise with standard deviation $0.08$ and (c) shows the reconstructed image.}
	\label{fig:ic_lichtenstein}	
\end{figure}	

We performed our numerical tests on the colored test image lichtenstein (see Figure \ref{fig:ic_lichtenstein}) of size $256\times 256$ making each color ranging in the closed interval from $0$ to $1$. 
By adding white Gaussian noise with standard deviation $0.08$, we obtained the noisy image $b\in\R^{n}$. We took $\omega_1=(1,1)$ and $\omega_2=(1,1)$, the regularization parameters in 
($\ell_2^2$-IC/P) and ($\ell_2^2$-MIC/P) were set to $\alpha_1=0.06$ and $\alpha_2=0.2$, while the tests were made on an Intel Core i$7$-$3770$ processor.

When measuring the quality of the restored images, we used the improvement in signal-to-noise ratio (ISNR), which is given by 
\begin{align*}
	\text{ISNR}_k = 10 \log_{10}\left( \frac{\|x-b\|^2}{\|x-x_k\|^2}\right),
\end{align*}  
where $x$, $b$, and $x_k$ are the original, the observed noisy and the reconstructed image at iteration $k \in \N$, respectively.

In Figure \ref{fig:ic_ISNR} we compare the performances of Algorithm \ref{ic_alg1_convex} (FB) and Algorithm \ref{ic_alg2_convex} (FBF) in the context of solving the optimization problems \eqref{ic_problem_IC} and \eqref{ic_problem_MIC}
to the ones of different optimization algorithms.
\begin{figure}[tb]	
	\centering
	\captionsetup[subfigure]{position=top}
	\subfloat[ISNR values for ($\ell_2^2$-IC/P)]{\includegraphics*[width=0.5\textwidth]{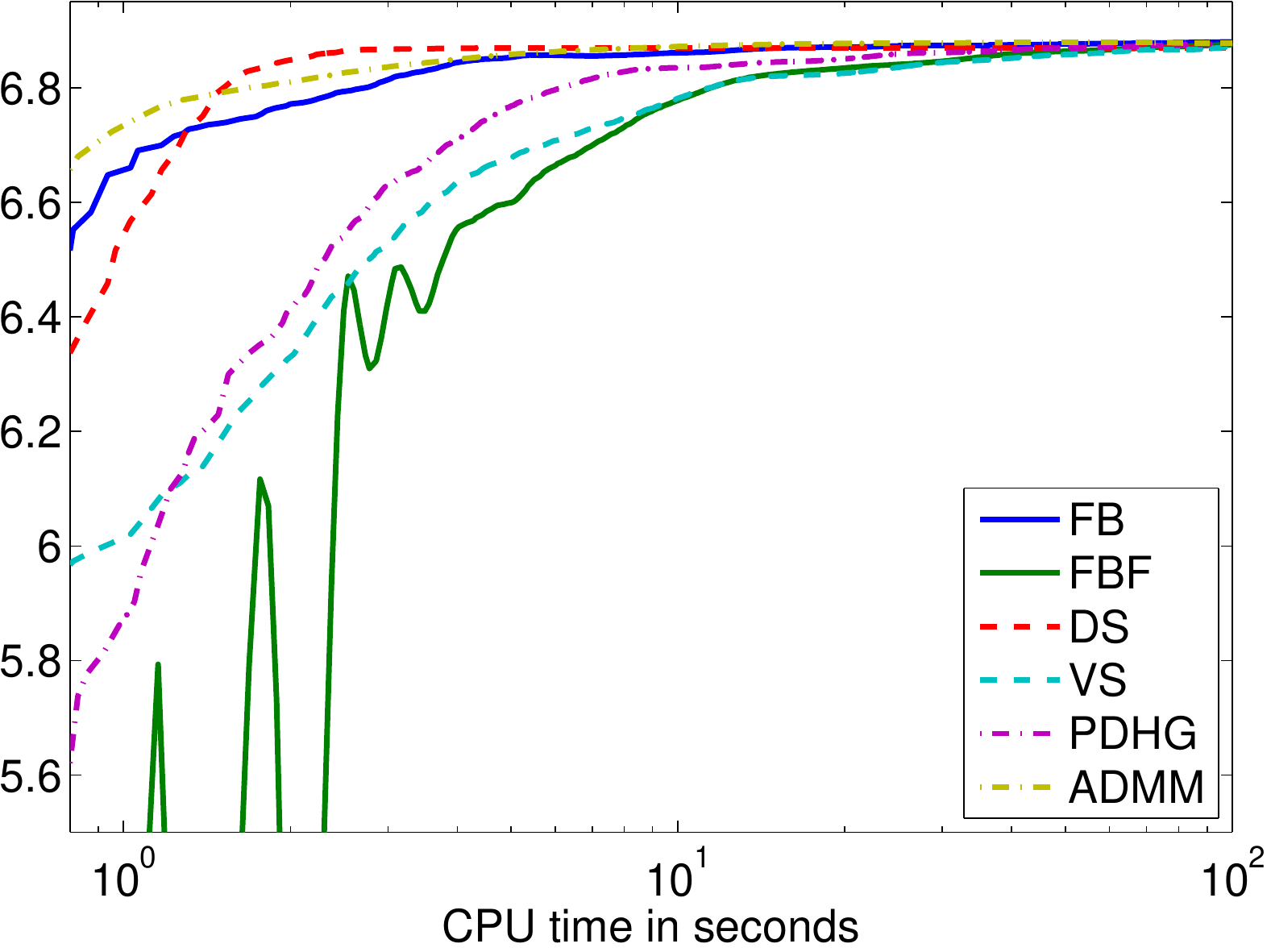}}
	\subfloat[ISNR values for ($\ell_2^2$-MIC/P)]{\includegraphics*[width=0.5\textwidth]{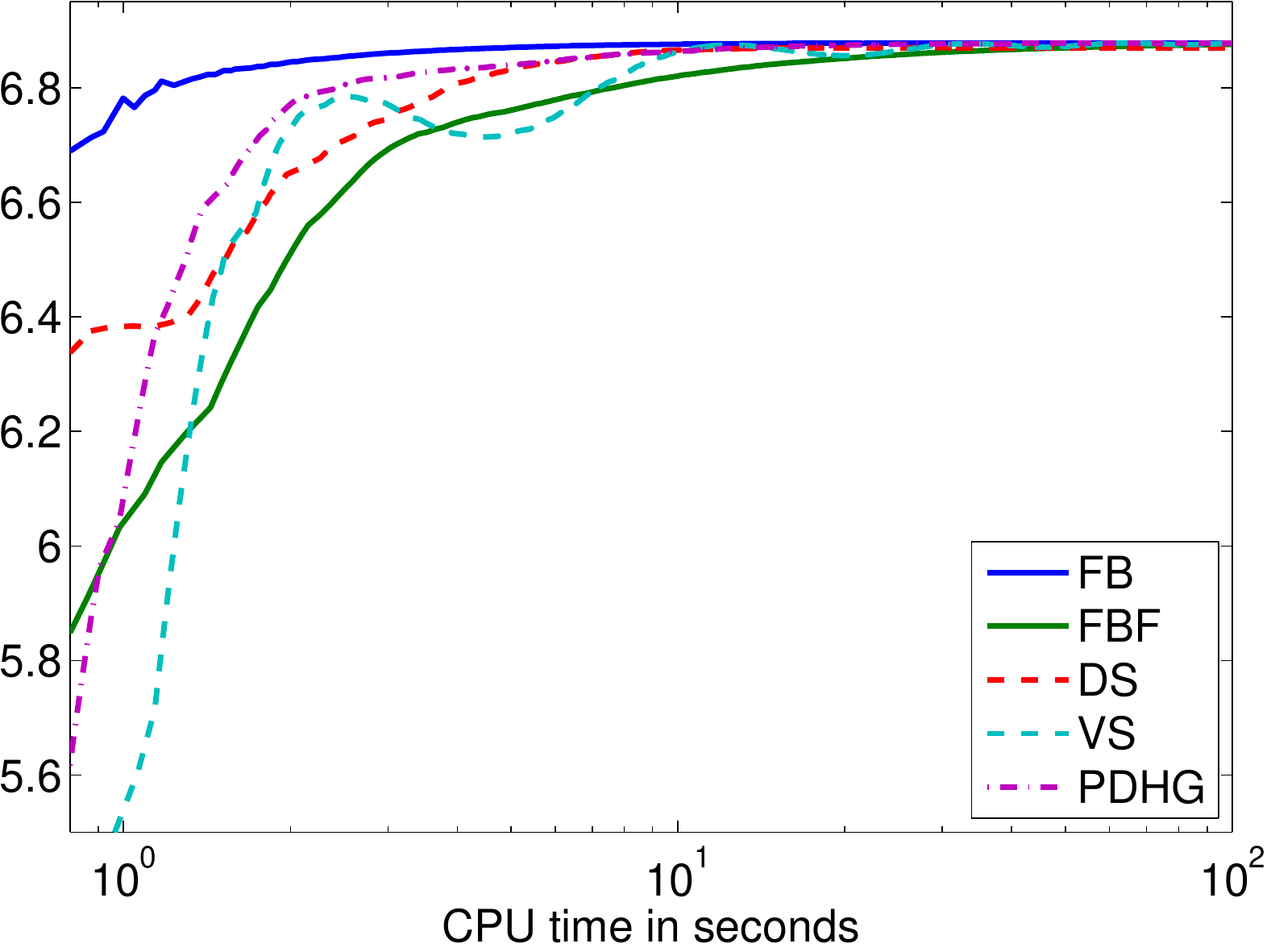}}
	\caption{\small Figure (a) shows the evolution of the ISNR for the ($\ell_2^2$-IC/P) problem w.r.t. the CPU times (in seconds) in log scale. Figure (b) shows the evolution of the ISNR for the ($\ell_2^2$-MIC/P) problem w.r.t. the CPU times (in seconds) in 
	log scale.}
	\label{fig:ic_ISNR}	
\end{figure}	

The double smoothing (DS) algorithm, as proposed in \cite{BotHend13a}, is applied to the Fenchel dual problems of \eqref{ic_problem_IC_2} and \eqref{ic_problem_MIC_2} by considering the acceleration strategies in \cite{BotHend12f}. 
One should notice that, since the smoothing parameters are constant, (DS) solves continuously differentiable approximations of \eqref{ic_problem_IC_2} and \eqref{ic_problem_MIC_2} and does therefore not necessarily 
converge to the unique minimizers of \eqref{ic_problem_IC} and \eqref{ic_problem_MIC}. As a second smoothing algorithm, we considered the variable smoothing technique (VS) in \cite{BotHend12c}, 
which successively reduces the smoothing parameter in each iteration and therefore solves the primal optimization problems as the iteration counter increases. 
We further considered the primal-dual hybrid gradient method (PDHG) as discussed in \cite{SetSteTeu11}, which is nothing else than the primal-dual method in \cite{ChaPoc11}. 
Finally, the alternating direction method of multipliers (ADMM) was applied to \eqref{ic_problem_IC_2}, as it was also done in \cite{SetSteTeu11}. 
Here, one makes use of the Moore-Penrose inverse of a special linear bounded operator which can be implemented in this setting efficiently, since $\D_1^T\D_1$ and $\D_2^T\D_2$ can be diagonalized by the discrete cosine transform. 
The problem which arises in \eqref{ic_problem_MIC_2}, however, is far more difficult to be solved with this method (and was therefore not implemented), since the linear bounded operator assumed to be inverted has a 
more complicated structure. This reveals a typical drawback of ADMM given by the fact that this method does not provide a full splitting, like primal-dual or smoothing algorithms do.

As it follows from the comparisons shown in Figure \ref{fig:ic_ISNR}, the FBF method suffers because of its additional forward step. 
However, many time-intensive steps in this algorithm could have been executed in parallel, which would lead to a significant decrease of the execution time. 
On the other hand, the FB method performs fast and stable in both examples, while optical differences in the reconstructions for ($\ell_2^2$-IC/P) and ($\ell_2^2$-MIC/P) are not observable.

\end{document}